\documentclass{article}
\usepackage{array,amscd,amsfonts,amsmath,amssymb,amsthm,graphics,mathrsfs,multirow,verbatim,enumitem,hyperref,xfrac}
\usepackage[utf8]{inputenc}
\usepackage[english]{babel}
\begingroup
    \makeatletter
    \@for\theoremstyle:=plain,definition,remark\do{%
        \expandafter\g@addto@macro\csname th@\theoremstyle\endcsname{%
            \addtolength\thm@preskip\parskip
            }%
        }
\endgroup
\usepackage[parfill]{parskip}
\usepackage{geometry}
\theoremstyle{plain}
\newtheorem{theorem}{Theorem}[section]
\newtheorem{corollary}[theorem]{Corollary}
\newtheorem{lemma}[theorem]{Lemma}
\newtheorem{proposition}[theorem]{Proposition}
\newtheorem{conjecture}[theorem]{Conjecture}
\theoremstyle{definition}
\newtheorem{definition}[theorem]{Definition}
\newtheorem{example}[theorem]{Example}
\newtheorem*{acknowledgements*}{Acknowledgements}
\theoremstyle{remark}
\newtheorem{remark}[theorem]{Remark}

\begin{document}
\sloppy

\title{\Large{\textbf{SMOOTH SURFACES IN SMOOTH FOURFOLDS WITH VANISHING FIRST CHERN CLASS}}}
\author{\normalsize{BENJAMIN E. DIAMOND}}
\date{}
\maketitle

\begin{abstract}
According to a conjecture attributed to Hartshorne and Lichtenbaum and proven by Ellingsrud and Peskine \cite{Ellingsrud:1989aa}, the smooth rational surfaces in $\mathbb{P}^4$ belong to only finitely many families. We formulate and study a collection of analogous problems in which $\mathbb{P}^4$ is replaced by a smooth fourfold $X$ with vanishing first integral Chern class. We embed such $X$ into a smooth ambient variety and count families of smooth surfaces which arise in $X$ from the ambient variety. We obtain various finiteness results in such settings. The central technique is the introduction of a new numerical invariant for smooth surfaces in smooth fourfolds with vanishing first Chern class.
\end{abstract}

\section{Introduction}
That those smooth complex projective algebraic varieties with vanishing first integral Chern class form a significant class has been understood for some time \cite{Beauville:1983aa}. This class includes the Calabi--Yau manifolds, which now occupy a central place in theoretical physics \cite{Candelas:1991aa, Greene:1995aa, Brunner:1997aa}, as well as, as a further special case, the hyper-K\"{a}hler varieties, which have proven a fertile testing ground for Bloch--Beilinson-type conjectures \cite{Beauville:2007aa, Voisin:2008aa, Ferretti:2012aa, Voisin:2016aa}. In this paper, we develop a family of techniques geared towards treating the fourfolds in this class---which, as the cited papers demonstrate, are of particular importance---and, in particular, the smooth surfaces inside them.

It was conjectured by Hartshorne and Lichtenbaum and proven by Ellingsrud and Peskine \cite{Ellingsrud:1989aa} in 1989 that the smooth surfaces not of general type in $\mathbb{P}^4$ have bounded degree. Ciliberto and Di Gennaro \cite{Ciliberto:2002aa} have generalized Ellingsrud and Peskine's result to more general smooth fourfolds, showing that for any such fourfold $X$ with fixed ample divisor, the smooth surfaces not of general type in $X$ have bounded degree whenever $X$ has Picard rank $\rho(X) = 1$. In this latter case, additional auxiliary arguments go on to demonstrate that this finiteness persists even when the surfaces are taken up to algebraic equivalence within the fourfold $X$ itself, as opposed to in an ambient projective space.

We develop a technique which treats the smooth surfaces in any smooth fourfold $X$ with vanishing first Chern class. Our technique, as above, counts families of smooth surfaces taken up to an adequate equivalence relation within the fourfold $X$ itself; in contrast with the above results, we count families not of non-general type smooth surfaces in $X$ but rather of smooth surfaces $S$ in $X$ whose Chern number expression $\text{deg} \left( c_1^2(\mathcal{T}_S) - c_2(\mathcal{T}_S) \right)$ attains some fixed value $s \in \mathbb{Z}$, where $s$ here is chosen freely. This expression detects a surface's Kodaira dimension, as well as, for example, its holomorphic Euler characteristic.

Central to the technique is the introduction of a new invariant for smooth surfaces $S$ in a smooth fourfold $X$ with vanishing first Chern class (see Proposition \ref{invariant}):
\begin{proposition}
The value of the Chern number expression $\emph{deg} \left( c_1^2(\mathcal{T}_S) - c_2(\mathcal{T}_S) \right)$ depends only on a smooth surface $S$ in $X$'s numerical equivalence class.
\end{proposition}

We furthermore facilitate the computation of this invariant by embedding $X$ into an ambient variety. We define a class of embeddings to which our theory applies (see Definition \ref{clean}):

\begin{definition}
We will say that a smooth fourfold $X$ with vanishing first Chern class is embedded in a smooth variety $V$ \textit{cleanly} if the second Chern class $c_2 \left( \mathcal{N}_{X/V} \right)$ of the normal bundle of $X$ in $V$ is the restriction to $X$ of a cycle class on $V$.
\end{definition}

To any clean embedding $X \subset V$ we associate a remarkable function defined on the group of codimension-2 cycles in $V$ up to numerical equivalence (see Definition \ref{associated}, as well as Proposition \ref{core} and Lemma \ref{quadratic}):

\begin{proposition}
Let $X \subset V$ be a clean embedding. Then there exists a function $Q_{X \subset V} \colon N^2(V) \rightarrow \mathbb{Z}$ with the properties that
\begin{enumerate}[label=\roman*., font=\itshape]
\item If a smooth surface $S \subset X$ satisfies $[S] = i^*(\alpha)$ for $\alpha \in N^2(V)$, then $Q_{X \subset V}(\alpha) = \emph{deg} \left( c_1^2(\mathcal{T}_S) - c_2(\mathcal{T}_S) \right)$.
\item Under an identification $N^2(V) \cong \mathbb{Z}^m$, the function $Q_{X \subset V}$ becomes quadratic with integral coefficients.
\end{enumerate}
\end{proposition}

We isolate a condition which controls the behavior of the function $Q_{X \subset V}$ (see Definition \ref{pair}):

\begin{definition}
We will call a clean embedding $X \subset V$ a \textit{decent pair} if the integral quadratic form $\alpha \mapsto \text{deg} ([X] \cdot \alpha \cdot \alpha)$ on $N^2(V)$ is positive definite.
\end{definition}

When $X \subset V$ is a decent pair, the asymptotic study of $Q_{X \subset V}$ yields (see Theorem \ref{bound}):

\begin{theorem} \label{imain}
Let $X \subset V$ be a decent pair. Then for any $s \in \mathbb{Z}$, at most finitely many numerical equivalence classes in $N^2(X)$ are representable by a smooth surface $S$ in $X$ satisfying $\emph{deg} \left( c_1^2(\mathcal{T}_S) - c_2(\mathcal{T}_S) \right) \leq s$ which arises in $X$ as a generically transverse intersection in $V$.
\end{theorem}

We finally develop a criterion for decency (see Proposition \ref{picdecent}):

\begin{proposition}
If $X$ is an intersection of ample divisors in $V$ and $\rho(V) = 1$, then $X \subset V$ is a decent pair.
\end{proposition}

We also develop versions of the theory for finer adequate equivalence relations.

This method serves to establish---or recover---the non-rationality of many smooth surfaces $S$ in $X$. We also produce new results in the flavor of that of Ciliberto and Di Gennaro. For example (see Theorem \ref{cdg}):

\begin{theorem}
Consider a smooth fourfold $X$ with vanishing first Chern class and Picard rank $1$. Then for any $r \in \mathbb{Z}$, at most finitely many numerical equivalence classes in $N^2(X)$ are representable by a smooth surface $S$ in $X$ satisfying $\chi(S, \mathcal{O}_S) \leq r$.
\end{theorem}

Specializing the result of Theorem \ref{imain} to example pairs $X \subset V$ for which the associated function $Q_{X \subset V}$ can be determined exactly, we enrich its finiteness assertions so as to specify concrete numerical bounds. These bounds are in each case obtained through lattice point counting techniques of varied and interesting forms. For example (see Theorems \ref{sextic}, \ref{fano}, and \ref{130}):

\begin{theorem}
Consider the smooth sextic fourfold $X \subset \mathbb{P}^5$. Let $r \in \mathbb{Z}$. Then at most
\begin{equation*}\frac{\sqrt{90^2 + 144r}}{6} + 1\end{equation*}
elements of $CH^2(X)$ are representable by a smooth surface $S$ in $X$ satisfying $\chi(S, \mathcal{O}_S) \leq r$ which arises in $X$ as a generically transverse intersection in $\mathbb{P}^5$.
\end{theorem}

\begin{theorem}
Consider the Fano variety $F \subset G(2,6)$ of lines in a general cubic fourfold. Let $r \in \mathbb{Z}$. Then at most
\begin{equation*}\frac{\pi(6r + 207)}{\sqrt{891}} + 8 + 8 \cdot 2 \sqrt{\frac{6r + 207}{9(4 - \sqrt{5})}}\end{equation*}
elements of $CH^2(F)$ are representable by a smooth surface $S$ in $F$ satisfying $\chi(S, \mathcal{O}_S) \leq r$ which arises in $F$ as a generically transverse intersection in $G(2,6)$.
\end{theorem}

\begin{theorem}
Consider the smooth Calabi--Yau fourfold $X \subset \mathbb{P}^4 \times \mathbb{P}^6$ of \cite[\#130]{Gray:2013aa}. Let $r, q \in \mathbb{Z}$. Then at most finitely many elements of $CH^2(X)$ are representable by a smooth surface $S$ in $X$ satisfying $\chi(S, \mathcal{O}_S) \leq r$ and $K_S^2 \geq q$ which arises in $X$ as a generically transverse intersection in $\mathbb{P}^4 \times \mathbb{P}^6$.
\end{theorem}

\begin{acknowledgements*}
I would like to thank Claire Voisin for suggesting the problem that originally led to this work, and for many instructive comments. I would like to thank Caterina Consani for her advice. Vincenzo Di Gennaro, Robert Laterveer, John Ottem, Steven Sam, David Savitt, Vyacheslav Shokurov, and an anonymous referee also provided helpful answers and suggestions.\end{acknowledgements*}

\subsection{Notations and terminology}

We adopt notation similar to that of Eisenbud and Harris \cite{Eisenbud:2016aa}.

A \textit{scheme} will be a separated scheme of finite type over $\mathbb{C}$. A \textit{variety} will be an integral projective scheme over $\mathbb{C}$. \textit{Surfaces} and \textit{fourfolds} will be varieties in this sense (of dimensions 2 and 4, respectively), as will be \textit{subvarieties}. Subvarieties and \textit{embeddings} will always be closed. \textit{Smooth} will mean smooth over the base field $\mathbb{C}$, and smoothness will be mentioned explicitly when assumed.

We work primarily with algebraic cycles up to numerical equivalence. We have the groups $N_k(X)$ of $k$-dimensional cycles up to numerical equivalence in a smooth variety $X$, defined say as in Fulton \cite[Def. 19.1]{Fulton:1984aa}. The groups $N_k(X)$ are finitely generated free abelian groups (see \cite[19.3.2. (i)]{Fulton:1984aa}). Identifying $N^k(X) = N_{\text{dim}X - k}(X)$, we have the ring $N^*(X) = \bigoplus_{k = 0}^{\text{dim}X} N^k(X)$ of cycle classes up to numerical equivalence in $X$, introduced for example in \cite[\S C.3.3]{Eisenbud:2016aa}. We write $[S]$ for the cycle class up to numerical equivalence associated to a closed subscheme $S \subset X$, defined say as in \cite[\S 1.2.1]{Eisenbud:2016aa}. We say that subvarieties $A$ and $B$ of $X$ intersect \textit{generically transversally} if at a general point $p$ of each component $C$ of $A \cap B$, $A$, $B$, and $X$ are smooth at $p$ and the tangent spaces $T_pA$ and $T_pB$ at $p$ span $T_pX$ (see \cite[p. 18]{Eisenbud:2016aa}). $N^*(X)$ has a graded ring structure with the property that when subvarieties $A$ and $B$ of $X$ intersect generically transversally, $[A] \cdot [B] = [A \cap B]$.

We consider now a closed embedding of smooth varieties $i \colon X \rightarrow V$. The pushforward homomorphisms on Chow groups defined say in \cite[Def. 1.19]{Eisenbud:2016aa} descend here to numerical equivalence, by say \cite[Ex. 19.1.6]{Fulton:1984aa}, so that we have \textit{pushforward} homomorphisms $i_* \colon N_k(X) \rightarrow N_k(V)$, defined by declaring for any subvariety $A$ of $X$ that $i_* \colon [A] \mapsto [i(A)]$. We also have the \textit{degree} homomorphism $\text{deg} \colon N_0(X) \rightarrow \mathbb{Z}$, defined by assigning $\deg \colon [p] \mapsto 1$ for any closed point $p$ in $X$. The pullback homomorphism on Chow groups defined say in \cite[Thm. 1.23]{Eisenbud:2016aa} descends also to numerical equivalence (see \cite[Ex. 19.2.3]{Fulton:1984aa}), so that we have a \textit{pullback} homomorphism of graded rings $i^* \colon N^*(V) \rightarrow N^*(X)$, which acts by intersection with $X$, in the sense that whenever a subvariety $S' \subset V$ is such that $i^{-1}(S')$ is of the expected codimension and generically reduced, $i^*([S']) = [i^{-1}(S')]$. We finally have the \textit{push-pull formula}, which implies in particular that for any element $\alpha$ of $N^*(V)$, $i_*i^*(\alpha) = [X] \cdot \alpha$ in $N^*(V)$ (see \cite[p. 31]{Eisenbud:2016aa}).

We use the definition of the \textit{total Chern class} $c(\mathcal{E}) = 1 + c_1(\mathcal{E}) + c_2(\mathcal{E}) + \cdots \in N^*(X)$ of a vector bundle $\mathcal{E}$ on a smooth variety $X$ given in \cite[Thm. 5.3]{Eisenbud:2016aa}. Particularly important is the \textit{Whitney sum formula}, which declares that $c(\mathcal{E}) \cdot c(\mathcal{G}) = c(\mathcal{F})$ for any exact sequence of vector bundles $0 \rightarrow \mathcal{E} \rightarrow \mathcal{F} \rightarrow \mathcal{G} \rightarrow 0$ on $X$, as well as the fact that if global sections $\tau_0, \ldots , \tau_{r - k}$ of a bundle $\mathcal{E}$ of rank $r$ on $X$ become linearly dependent on a locus $D$ of codimension $k$ in $X$, then $[D] = c_k(\mathcal{E}) \in N^k(X)$ (see \cite[Thm. 5.3. (c), (b)]{Eisenbud:2016aa}). By \textit{the $k$\textsuperscript{th} Chern class of X} we shall mean the Chern class $c_k \left( \mathcal{T}_X \right)$ of $X$'s tangent bundle.

We use the symbol $\mathbb{N}$ to denote the nonnegative integers $\mathbb{Z}_{\geq 0}$.

That a smooth variety $X$ is \textit{Calabi--Yau} will mean that its canonical bundle $K_X$ is trivial. The equalities $0 = c_1(\mathcal{O}_X) = c_1(K_X) = c_1(\Omega_X) = -c_1(\mathcal{T}_X)$ indicate that a smooth Calabi--Yau variety $X$ has vanishing first Chern class. That a smooth variety $X$ is \textit{hyper-K\"{a}hler} will mean that the vector space $H^{2,0}(X)$ of holomorphic 2-forms on $X$ is generated over $\mathbb{C}$ by a single everywhere-nondegenerate 2-form $\sigma$. A hyper-K\"{a}hler variety $X$ is in particular Calabi--Yau, as the top exterior power of the 2-form $\sigma$ trivializes its canonical bundle.

\section{Results}

\subsection{A new invariant}

We describe a new invariant for smooth surfaces $S$ in a smooth fourfold $X$ with vanishing first Chern class. The vanishing of $c_1 \left( \mathcal{T}_X \right)$ relates the self-intersection number of a smooth surface $S \subset X$ to an expression which depends on $S$ alone and not on its embedding in $X$, in the sense that the only terms in the expression which involve $S$'s Chern classes are its Chern numbers.

\begin{proposition} \label{invariant}
The value of the Chern number expression $\emph{deg} \left( c_1^2(\mathcal{T}_S) - c_2(\mathcal{T}_S) \right)$ depends only on a smooth surface $S$ in $X$'s numerical equivalence class.
\end{proposition}
\begin{proof}
We use the self-intersection formula of Mumford (see Hartshorne \cite[\S A. 3. C7]{Hartshorne:1977aa}). We denote by $j$ the inclusion of $S$ into $X$.

The normal bundle exact sequence of $S$ in $X$ yields the following formulae for the Chern classes of $\mathcal{N}_{S/X}$, where in (\ref{adjunction1}) we use the vanishing of $c_1 \left( \mathcal{T}_X \right)$:
\begin{align}
c_1 \left( \mathcal{N}_{S/X} \right) &= c_1 \left( \left. \mathcal{T}_X \right|_S \right) - c_1 \left( \mathcal{T}_S \right) = - c_1 \left( \mathcal{T}_S \right), \label{adjunction1} \\
c_2 \left( \mathcal{N}_{S/X} \right) &= c_2 \left( \left. \mathcal{T}_X \right|_S \right) - c_2 \left( \mathcal{T}_S \right) - c_1 \left( \mathcal{T}_S \right) \cdot c_1 \left( \mathcal{N}_{S/X} \right) = c_2 \left( \left. \mathcal{T}_X \right|_S \right) + c_1^2 \left( \mathcal{T}_S \right) - c_2 \left( \mathcal{T}_S \right). \label{adjunction2}
\end{align}
The self-intersection formula now gives:
\begin{align*}
\text{deg} \left( [S] \cdot [S] \right) &= \text{deg} \left( j_* \left( c_2 \left( \mathcal{N}_{S/X} \right) \right) \right) &\text{(by the self-intersection formula)} \\
&= \text{deg} \left( j_* j^* \left( c_2 \left( \mathcal{T}_X \right) \right) \right) + \text{deg} \left( c_1^2 \left( \mathcal{T}_S \right) - c_2 \left( \mathcal{T}_S \right) \right) &\text{(by equation (\ref{adjunction2}) above)} \\
&= \text{deg} \left( [S] \cdot c_2 \left( \mathcal{T}_X \right) \right) + \text{deg} \left( c_1^2 \left( \mathcal{T}_S \right) - c_2 \left( \mathcal{T}_S \right) \right). &\text{(by the push-pull formula)}
\end{align*}
We thus establish the equality:
\begin{equation*}\text{deg} \left( c_1^2 \left( \mathcal{T}_S \right) - c_2 \left( \mathcal{T}_S \right) \right) = \text{deg} \left( [S] \cdot [S] \right) - \text{deg} \left( [S] \cdot c_2 \left( \mathcal{T}_X \right) \right).\end{equation*}
This equation's right-hand side depends only on $S$'s numerical equivalence class in $X$.
\end{proof}

\begin{remark}
In fact, an analogous invariant exists for the smooth half-dimensional subvarieties $S$ in a smooth variety $X$ of any even dimension $2d$, provided that $X$'s first $d - 1$ Chern classes vanish. We record the resulting invariant expressions for various low values of $d$:
\begin{enumerate}
\item $\text{deg} \left( -c_1(\mathcal{T}_S) \right),$
\item $\text{deg} \left( c_1^2(\mathcal{T}_S) - c_2(\mathcal{T}_S) \right),$
\item $\text{deg} \left( - c_1^3(\mathcal{T}_S) + 2c_1c_2(\mathcal{T}_S) - c_3(\mathcal{T}_S) \right),$
\item $\text{deg} \left( c_1^4(\mathcal{T}_S) - 3c_1^2c_2(\mathcal{T}_S) + 2c_1c_3(\mathcal{T}_S) + c_2^2(\mathcal{T}_S) - c_4(\mathcal{T}_S) \right).$
\end{enumerate}
These hold, for example, if $X$ is say an abelian variety, all of whose positive Chern classes necessarily vanish (see Mumford \cite[\S 4, Ques. 4. (iii)]{Mumford:1985aa}).

We decline to pursue this additional direction in what follows.
\end{remark}

\subsection{Ambient surfaces and associated functions}

In practice, the invariant of Proposition \ref{invariant} is computed by embedding $X$ into an ambient variety $V$.

Indeed, though the invariance of the expression $\text{deg} \left( c_1^2(\mathcal{T}_S) - c_2(\mathcal{T}_S) \right)$ up to numerical equivalence is established in the absence of an ambient variety, the computation of the value of this invariant on any particular smooth surface $S$ in $X$ is feasible only when the relevant intersection-theoretic calculations can be outsourced to a variety $V$ whose intersection ring is completely understood. (This requirement typically goes unmet by the smooth fourfold $X$ itself.) We develop this theory in what follows.

We introduce a key technical condition on embeddings:

\begin{definition} \label{clean}
Let $X$ be a smooth fourfold with vanishing first Chern class, embedded in a smooth variety $V$. We will say that the embedding of $X$ in $V$ is \textit{clean} if the second Chern class $c_2 \left( \mathcal{N}_{X/V} \right)$ of the normal bundle of $X$ in $V$ is the restriction to $X$ of a cycle class on $V$.
\end{definition}

For example, $X$ is cleanly embedded in $V$ if any of the following is true:
\begin{enumerate}
\item \label{clean1} The normal bundle of $X$ in $V$ is the restriction to $X$ of a bundle on $V$. (We use $c_2 \left( \mathcal{N}_{X/V} \right) = c_2(\left. \mathcal{E} \right|_X) = i^*(c_2(\mathcal{E}))$.)
\item \label{clean2} $X$ is a complete intersection in $V$. (\ref{clean1}. above holds in this case by the adjunction formula.)
\item \label{clean3} $X$ is defined in $V$ as the zero locus of expected dimension of a map between vector bundles on $V$. (This generalization of \ref{clean2}. above appears in, for example, Harris and Tu \cite[\S 3]{Harris:1984aa}.)
\item $\text{codim}_V(X) \leq 2$. (If $\text{codim}_V(X) = 0$ this is trivial; if $\text{codim}_V(X) = 1$ then $\mathcal{N}_{X/V}$ is a line bundle and $c_2 \left( \mathcal{N}_{X/V} \right) = 0$; if $\text{codim}_V(X) = 2$ then we use $c_2 \left( \mathcal{N}_{X/V} \right) = i^*(i_*([X])) = i^*([X])$ (see \cite[\S A. 3. C7]{Hartshorne:1977aa}).)
\item $X$ is an abelian variety. (In this case each $c_k(\mathcal{T}_X) = 0$ (see \cite[\S 4, Ques. 4. (iii)]{Mumford:1985aa}), and the normal bundle exact sequence of $X$ in $V$ shows that $c_2 \left( \mathcal{N}_{X/V} \right) = c_2 \left( \left. \mathcal{T}_V \right|_X \right) = i^* \left( c_2 \left( \mathcal{T}_V \right) \right)$.)
\end{enumerate}

\begin{remark}
Unfortunately, it appears that an arbitrary degeneracy locus of expected dimension (that is, one defined by a rank condition which is not that of zero rank) is not in general embedded cleanly. Though constructing a concrete example appears difficult, we observe analogous behavior in Segre varieties (see \cite[(1.8) Rem.]{Harris:1984aa}).
\end{remark}

We define a class of surfaces to which our theory applies:

\begin{definition} \label{ambient}
Let $X$ be embedded cleanly in $V$, and denote by $i$ the inclusion. We shall say that a surface $S$ in $X$ is \textit{ambient in $V$}, or \textit{ambient}, if $[S] = i^*(\alpha)$ for some cycle class $\alpha \in N^2(V)$.
\end{definition}

For example, a surface $S \subset X$ is ambient if any of the following is true:
\begin{enumerate}
\item \label{ambient1} $S = i^{-1}(S')$ for some subvariety $S'$ of codimension 2 in $V$. (In this case $i^{-1}(S')$ is of the expected codimension and generically reduced, and we use \cite[Thm. 1.23. (a)]{Eisenbud:2016aa}.)
\item \label{ambient2} $S$ is the dependency locus of the expected codimension 2 of sections $\tau_0, \ldots , \tau_{r - 2}$ of a rank-$r$ bundle $\left. \mathcal{E} \right|_X$ on $X$ which is the restriction to $X$ of a bundle on $V$. (This follows from $[S] = c_2(\left. \mathcal{E} \right|_X) = i^*(c_2(\mathcal{E}))$.)
\end{enumerate}

\begin{remark}
Such $S$ are in fact ambient even over rational equivalence (see Definition \ref{eambient} below).
\end{remark}

An example of a non-ambient surface is given in Example \ref{nonambient} below; the demonstration that this surface is not ambient, however, relies on the tools developed in this section.

We have the following property of clean embeddings $X \subset V$:

\begin{lemma} \label{restriction}
Let $X \subset V$ be a clean embedding. Then the Chern class $c_2 \left( \mathcal{T}_X \right) \in N^2(X)$ is the restriction to $X$ of a cycle class on $V$.
\end{lemma}
\begin{proof} The normal bundle sequence of $X$ in $V$ gives that
\begin{align*}
c_2 \left( \mathcal{T}_X \right) &= c_2 \left( \left. \mathcal{T}_V \right|_X \right) - c_1 \left( \mathcal{T}_X \right) \cdot c_1 \left( \mathcal{N}_{X/V} \right) - c_2 \left( \mathcal{N}_{X/V} \right) \\
&= i^* \left( c_2 \left( \mathcal{T}_V \right) \right) - c_2 \left( \mathcal{N}_{X/V} \right).
\end{align*}
The cleanness of $X \subset V$ asserts that the right-hand term is the restriction to $X$ of a cycle on $V$.
\end{proof}

We will decline to distinguish, in this situation, between $c_2 \left( \mathcal{T}_X \right)$ and the cycle class in $N^2(V)$ which restricts to it.

By Lemma \ref{restriction} above, the following definition makes sense:

\begin{definition} \label{associated}
Let $X$ be a smooth fourfold with vanishing first Chern class, embedded cleanly in a smooth variety $V$. We define the \textit{function associated to the embedding} $Q_{X \subset V} \colon N^2(X) \rightarrow \mathbb{Z}$ by associating to any cycle class $\alpha \in N^2(V)$ the intersection number:
\begin{equation*}Q_{X \subset V}(\alpha) := \text{deg} \left( [X] \cdot \alpha \cdot \alpha - [X] \cdot \alpha \cdot c_2 \left( \mathcal{T}_{X} \right) \right).\end{equation*}
\end{definition}

The following is the technical core of the paper:

\begin{proposition} \label{core}
Suppose that a smooth surface $S \subset X$ is ambient, with $[S] = i^*(\alpha)$. Then
\begin{equation*}Q_{X \subset V}(\alpha) = \emph{deg} \left( c_1^2(\mathcal{T}_S) - c_2(\mathcal{T}_S) \right).\end{equation*}
\end{proposition}
\begin{proof}
We expand upon the proof of Proposition \ref{invariant}. We again denote by $j$ the inclusion of $S$ into $X$.

The self-intersection formula and several applications of the push-pull formula give:
\begin{align*}
\text{deg} \left( [S] \cdot [S] \right) &= \text{deg} \left( j_* \left( c_2 \left( \mathcal{N}_{S/X} \right) \right) \right) &\text{(by the self-intersection formula)} \\
&= \text{deg} \left( j_* j^* \left( c_2 \left( \mathcal{T}_X \right) \right) \right) + \text{deg} \left( c_1^2 \left( \mathcal{T}_S \right) - c_2 \left( \mathcal{T}_S \right) \right) &\text{(by equation (\ref{adjunction2}) above)} \\
&= \text{deg} \left( [S] \cdot c_2 \left( \mathcal{T}_X \right) \right) + \text{deg} \left( c_1^2 \left( \mathcal{T}_S \right) - c_2 \left( \mathcal{T}_S \right) \right) &\text{(by the push-pull formula)} \\
&= \text{deg} \left( i_* i^* \left( \alpha \cdot c_2 \left( \mathcal{T}_X \right) \right) \right) + \text{deg} \left( c_1^2 \left( \mathcal{T}_S \right) - c_2 \left( \mathcal{T}_S \right) \right) &\text{(by $[S] = i^*(\alpha)$ and Lemma \ref{restriction})} \\
&= \text{deg} \left( [X] \cdot \alpha \cdot c_2 \left( \mathcal{T}_X \right) \right) + \text{deg} \left( c_1^2 \left( \mathcal{T}_S \right) - c_2 \left( \mathcal{T}_S \right) \right). &\text{(by the push-pull formula)}
\intertext{On the other hand, we also have:}
\text{deg} \left( [S] \cdot [S] \right) &= \text{deg} \left( i_* i^* (\alpha \cdot \alpha) \right) &\text{(using $[S] = i^*(\alpha)$)} \\
&= \text{deg} \left( [X] \cdot \alpha \cdot \alpha \right). &\text{(by the push-pull formula)}
\end{align*}
The concluding lines of the above two calculations complete the proof, by definition of $Q_{X \subset V}$.
\end{proof}

\subsection{Background in the theory of smooth surfaces}

We recall notions from the well-established theory of smooth surfaces. We refer to the text of Barth, Hulek, Peters, and Van de Ven \cite{Barth:2004aa}. In particular, we recall Noether's formula (see the case $n = 2$ following \cite[I, (5.5) Thm.]{Barth:2004aa}), the Gauss--Bonnet theorem (see \cite[p. 23]{Barth:2004aa}), and the existence of minimal models (see \cite[III, (4.5) Thm.]{Barth:2004aa}).

We let $S$ be a smooth surface in what follows.

\begin{lemma} \label{blow}
Consider the blowing-up $\sigma \colon \bar{S} \rightarrow S$ of $S$ at a point. Then
\begin{equation*}\emph{deg} \left( c_1^2(\mathcal{T}_S) - c_2(\mathcal{T}_S) \right) = \emph{deg} \left( c_1^2(\mathcal{T}_{\bar{S}}) - c_2(\mathcal{T}_{\bar{S}}) \right) + 2.\end{equation*}
\end{lemma}
\begin{proof}
From the isomorphisms $\sigma^* \colon H^i(S, \mathcal{O}_S) \rightarrow H^i(\bar{S}, \mathcal{O}_{\bar{S}})$ (see \cite[I, (9.1) Thm. (iii)]{Barth:2004aa}) and Noether's formula, it follows that $\text{deg} \left( c_1^2(\mathcal{T}_{\bar{S}}) + c_2(\mathcal{T}_{\bar{S}}) \right) = \text{deg} \left( c_1^2(\mathcal{T}_S) + c_2(\mathcal{T}_S) \right)$. The Gauss--Bonnet formula and the result of \cite[I, (9.1) Thm. (iv)]{Barth:2004aa} demonstrate that $\text{deg} \left( c_2(\mathcal{T}_{\bar{S}}) \right) = \text{deg} \left( c_2(\mathcal{T}_S) \right) + 1$. Combining these two observations completes the proof.
\end{proof}

\begin{proposition} \label{nongeneral}
Suppose that $S$ is not of general type. Then $\emph{deg} \left( c_1^2(\mathcal{T}_S) - c_2(\mathcal{T}_S) \right) \leq 6$.
\end{proposition}
\begin{proof}
By Lemma \ref{blow} and the result of \cite[I, (9.1) Thm. (viii)]{Barth:2004aa}, we may assume that $S$ is minimal. The result then follows from the classification of minimal surfaces, say as in \cite[VI, (1.1) Thm.]{Barth:2004aa}.
\end{proof}

\begin{proposition} \label{euler}
Suppose that $S$ satisfies $\chi(S, \mathcal{O}_S) = r$. Then $\emph{deg} \left( c_1^2(\mathcal{T}_S) - c_2(\mathcal{T}_S) \right) \leq 6r$.
\end{proposition}
\begin{proof}
As in the proof of the above proposition, we may immediately replace $S$ by its minimal model. If $S$ is not of general type, then the conclusion follows from the explicit classification \cite[VI, (1.1) Thm.]{Barth:2004aa}. Assuming now that $S$ is of general type, we apply the Bogomolov--Miyaoka--Yau inequality $\text{deg} \left( c_1^2(\mathcal{T}_S) \right) \leq 3 \cdot \text{deg} \left( c_2(\mathcal{T}_S) \right)$ (see \cite[VII, (4.1) Thm.]{Barth:2004aa}). We have:
\begin{align*}
\text{deg} \left( c_1^2(\mathcal{T}_S) - c_2(\mathcal{T}_S) \right) &\leq \text{deg} \left( \frac{1}{2} \cdot c_1^2(\mathcal{T}_S) + \frac{3}{2} \cdot c_2(\mathcal{T}_S) - c_2(\mathcal{T}_S) \right) &\text{(using the BMY inequality)} \\
&= 6 \cdot \chi(S, \mathcal{O}_S). &\text{(by Noether's formula)}
\end{align*}
This concluding expression completes the proof.
\end{proof}

\begin{proposition} \label{alternate}
Suppose that $S$ satisfies $\chi(S, \mathcal{O}_S) = r$ and $K_S^2 = q$, where $K_S^2$ is $S$'s canonical self-intersection number. Then $\emph{deg} \left( c_1^2(\mathcal{T}_S) - c_2(\mathcal{T}_S) \right) = -12r + 2q$.
\end{proposition}
\begin{proof}
Identifying $c_1^2(\mathcal{T}_S) = K_S^2$, by Noether's formula we have that $\text{deg} \left( c_1^2(\mathcal{T}_S) - c_2(\mathcal{T}_S) \right) = - \text{deg} \left( c_1^2(\mathcal{T}_S) + c_2(\mathcal{T}_S) \right) + 2q = -12r + 2q$.
\end{proof}

\subsection{Positivity and the growth of associated functions}

We return to a clean embedding $X \subset V$. We have the following perspective on $Q_{X \subset V}$:

\begin{lemma} \label{quadratic}
Identifying $N^2(V) \cong \mathbb{Z}^m$ via a basis $e_1, \ldots , e_m$ of cycle classes, the function $Q_{X \subset V}$ becomes a quadratic function on $\mathbb{Z}^m$ with integral coefficients.
\end{lemma}
\begin{proof}Writing any cycle class $\alpha \in N^2(V)$ uniquely as $\alpha = x_1 e_1 + \cdots + x_m e_m$, we have:
\begin{align*}Q_{X \subset V}(\alpha) &= \text{deg} \left( [X] \cdot \alpha \cdot \alpha - [X] \cdot \alpha \cdot c_2 \left( \mathcal{T}_{X} \right) \right) \\
&= \text{deg} \left( [X] \cdot (x_1 e_1 + \cdots + x_m e_m) \cdot ((x_1 e_1 + \cdots + x_m e_m) - c_2 \left( \mathcal{T}_{X} \right)) \right) \\
&= \sum_{i, j = 1}^m \text{deg} ([X] \cdot e_i \cdot e_j) \cdot x_i x_j - \sum_{k = 1}^m \text{deg} \left( [X] \cdot e_k \cdot c_2 \left( \mathcal{T}_{X} \right) \right) \cdot x_k.
\end{align*}
This concluding expression completes the proof.
\end{proof}

We remark that the second-order part of $Q_{X \subset V}$ is precisely $\alpha \mapsto \text{deg} ([X] \cdot \alpha \cdot \alpha)$.

We isolate an important positivity condition on clean embeddings $X \subset V$:

\begin{definition} \label{pair}
We will call a clean embedding $X \subset V$ a \textit{decent pair} if the integral quadratic form $\alpha \mapsto \text{deg} ([X] \cdot \alpha \cdot \alpha)$ on $N^2(V)$ is positive definite.
\end{definition}

By the basic results of the previous section, $Q_{X \subset V}$ controls the smooth ambient surfaces $S$ in $X$:

\begin{theorem} \label{bound}
Let $X \subset V$ be a decent pair. Then for any $s \in \mathbb{Z}$, at most finitely many numerical equivalence classes in $N^2(X)$ are representable by a smooth ambient surface $S$ in $X$ satisfying $\emph{deg} \left( c_1^2(\mathcal{T}_S) - c_2(\mathcal{T}_S) \right) \leq s$.
\end{theorem}
\begin{proof}
Proposition \ref{core} implies that any element $\alpha \in N^2(V)$ for which there exists a smooth surface $S \subset X$ with $\text{deg} \left( c_1^2(\mathcal{T}_S) - c_2( \mathcal{T}_S) \right) \leq s$ and $[S] = i^*(\alpha)$ satisfies $Q_{X \subset V}(\alpha) \leq s$. Definition \ref{pair} meanwhile implies that the second-order part of the quadratic function $Q_{X \subset V}$, say as in Lemma \ref{quadratic} above, is a positive definite quadratic form on $\mathbb{Z}^m$, and it follows that $Q_{X \subset V}(x_1, \ldots , x_m) \leq s$ for at most finitely many tuples $(x_1, \ldots , x_m) \in \mathbb{Z}^m$. The numerical equivalence classes $[S]$ of these surfaces are contained in the image under $i^*$ of this finite set.
\end{proof}

We let $X \subset V$ be a decent pair in the following corollaries:

\begin{corollary} \label{numnongeneral}
At most finitely many numerical equivalence classes in $N^2(X)$ are representable by a smooth ambient surface $S$ in $X$ not of general type.
\end{corollary}
\begin{proof}
By Proposition \ref{nongeneral}, this follows from Theorem \ref{bound} above using $s = 6$.
\end{proof}

\begin{corollary} \label{numrational}
At most finitely many numerical equivalence classes in $N^2(X)$ are representable by a smooth rational ambient surface $S$ in $X$.
\end{corollary}

\begin{corollary} \label{numeuler}
For any $r \in \mathbb{Z}$, at most finitely many numerical equivalence classes in $N^2(X)$ are representable by a smooth ambient surface $S$ in $X$ satisfying $\chi(S, \mathcal{O}_S) \leq r$.
\end{corollary}
\begin{proof}
Using Proposition \ref{euler}, we apply Theorem \ref{bound} using $s = 6r$.
\end{proof}

We have the following fundamental criterion for decency:

\begin{proposition} \label{picdecent}
Consider a clean embedding $X \subset V$ for which $[X] = \left( \prod_{k = 1}^{\emph{dim}(V) - 4} c_1(L_k) \right)$ for ample line bundles $L_k$ on $V$. Then $X \subset V$ is a decent pair if $\rho(V) = 1$.
\end{proposition}
\begin{proof}
We apply a specialization of the Hodge index theorem, together with a ``mixed'' variant of the Hodge--Riemann bilinear relations due to Timorin \cite{Timorin:1998aa} as well as Dinh and Nguy\^{e}n \cite{Dinh:2006aa}. We take all cycles as elements of $H^{2*}(V, \mathbb{Q})$ in what follows (see Voisin \cite[\S 2.1.4]{Voisin:2014aa}).

Fixing a K\"{a}hler form $\omega := c_1(\mathcal{O}_V(1))$ on $V$, for each $i \geq 0$ we define:
\begin{equation*}\textstyle H^{2-i,2-i}(V, \mathbb{Q})_{\text{prim}} := \text{ker} \left( \alpha \mapsto \alpha \wedge \left( \bigwedge_{k = 1}^{\text{dim}(V) - 4} c_1(L_k) \right) \wedge \omega^{2i + 1} \right).\end{equation*}
By \cite{Dinh:2006aa}, for each $i \geq 0$ the form $\left( \bigwedge_{k = 1}^{\text{dim}(V) - 4} c_1(L_k) \right) \wedge \omega^{2i}$ satisfies the Lefschetz decomposition and Hodge--Riemann bilinear relations for the bidegree $(2-i,2-i)$.

Denoting by $L$ the Lefschetz operator $L \colon H^k(V, \mathbb{Q}) \rightarrow H^{k + 2}(V, \mathbb{Q})$ on the cohomology of $V$, we have a generalized Lefschetz decomposition on $H^{2,2}(V, \mathbb{Q})$:
\begin{align*}
H^{2,2}(V, \mathbb{Q}) &= \bigoplus_{i \geq 0} L^i(H^{2-i,2-i}(V, \mathbb{Q})_{\text{prim}}) \\
&= H^{2,2}(V, \mathbb{Q})_{\text{prim}} \oplus L(H^{1,1}(V, \mathbb{Q})_{\text{prim}}) \oplus L^2(H^{0,0}(V, \mathbb{Q})_{\text{prim}}).
\end{align*}
Because the forms $\bigwedge_{k = 1}^{\text{dim}(V) - 4} c_1(L_k)$ and $[X]$ differ by a numerically trivial cycle, we may replace the former by the latter in what follows. We see that the pairing $\alpha \mapsto \int_V \alpha \wedge \overline{\alpha} \wedge [X]$ on $H^{2,2}(V, \mathbb{Q})$ is orthogonal across the above summands, and definite of sign $(-1)^{2 - i}$ on the $i$\textsuperscript{th} summand.

The Lefschetz theorem on $(1,1)$-classes on the other hand implies that $H^{1,1}(V, \mathbb{Q}) = H^{1,1}(V, \mathbb{Q})_{\text{prim}} \oplus L(H^{0,0}(V, \mathbb{Q})_{\text{prim}})$ is exhausted by its subgroup consisting of rational divisor classes. The assumption $\rho(V) = 1$ thus ensures that this subgroup consists of $L(H^{0,0}(V, \mathbb{Q})_{\text{prim}})$ alone, and that $H^{1,1}(V, \mathbb{Q})_{\text{prim}} = 0$. Thus $\alpha \mapsto \int_V \alpha \wedge \overline{\alpha} \wedge [X]$ is positive definite on $H^{2,2}(V, \mathbb{Q})$.

As $N^2(V)$ is torsion free, the quotient map $H^{2,2}(V, \mathbb{Z})_{\text{alg}} \rightarrow N^2(V)$ factors into a chain of quotients:
\begin{equation*}H^{2,2}(V, \mathbb{Z})_{\text{alg}} \rightarrow H^{2,2}(V, \mathbb{Z})_{\text{alg}} / \text{tors} \rightarrow N^2(V).\end{equation*}
Identifying the middle group with the image of $H^{2,2}(V, \mathbb{Z})_{\text{alg}} \rightarrow H^{2,2}(V, \mathbb{Q})$, we realize $N^2(V)$ as a subquotient of the abelian group $H^{2,2}(V, \mathbb{Q})$. Thus the positivity of the pairing on $H^{2,2}(V, \mathbb{Q})$ is induced also on $N^2(V)$.
\end{proof}

\begin{remark} \label{picdecentgen}
One is tempted to formulate the analogue of Proposition \ref{picdecent} in which $[X] = c_r(\mathcal{E})$ for any say globally generated vector bundle $\mathcal{E}$ of rank $r$ on $V$ (cf. Voisin \cite[\S 1, Ex. 2. (b)]{Voisin:2003aa}). Unfortunately, this more general statement appears out of reach.
\end{remark}

\begin{remark} \label{picdecentopp}
The opposite implication of Proposition \ref{picdecent} as well as of its generalization in Remark \ref{picdecentgen} would immediately follow upon assuming the Hodge conjecture for codimension-2 cycles on $V$. Indeed, $\rho(V) = 1$ if and only if the pairing $\alpha \mapsto \int_V \alpha \wedge \overline{\alpha} \wedge [X]$ is positive definite on $H^{2,2}(V, \mathbb{Q})$. Using now that the quotient map $H^{2,2}(V, \mathbb{Z})_{\text{alg}} \rightarrow N^2(V)$ annihilates only torsion (see \cite[19.3.2. (iii)]{Fulton:1984aa}), we see that $N^2(V)$ is precisely the subgroup of the image of $H^{2,2}(V, \mathbb{Z})$ in $H^{2,2}(V, \mathbb{Q})$ consisting of algebraic cycles. If this subgroup is of full rank, then the signature of the pairing on $H^{2,2}(V, \mathbb{Q})$ is inherited also on $N^2(V)$.
\end{remark}

The predictions of these remarks are validated in each of the case studies treated in Section \ref{detailed}.

\section{Decency over finer equivalence relations on cycles}

In this section, we develop analogues of the theory over equivalence relations on cycles finer than numerical equivalence. We then apply these results to ``tautological'' clean embeddings $X \subset X$, for which the Chow groups $CH^2(V) = CH^2(X)$ tend to be difficult to control.

Ciliberto and Di Gennaro \cite{Ciliberto:2002aa} prove that in any smooth fourfold $X$ with Picard rank $\rho(X) = 1$ (with fixed ample divisor), the smooth non-general type surfaces $S$ in $X$ have bounded degree. It follows from this by Kleiman \cite[Cor. 6.11. (ii)]{Kleiman:1971aa} that in fact these smooth surfaces $S$ have only finitely many Hilbert polynomials, and thus represent only finitely many components of the Hilbert scheme parameterizing the smooth surfaces in $X$. Because the Hilbert scheme is projective, it follows in turn that these $S$ belong to finitely many algebraic families. From say \cite[Ex. 10.3.3]{Fulton:1984aa} it follows finally that the $S$ represent only finitely many cycle classes up to algebraic equivalence. (I would like to thank Vincenzo Di Gennaro for this argument.)

The conclusions established in this section evoke, and sometimes refine, the specialization of this one to those fourfolds which in addition have vanishing first Chern class.

\subsection{Ambience and positivity over finer relations}

The general theory developed above ``lifts'' to finer equivalence relations on algebraic cycles.

We write $E$ for an adequate equivalence relation in the sense of say Jannsen \cite[p. 228]{Jannsen:2000aa}, so that for a smooth variety $X$ the ring of cycle classes in $X$ up to $E$-equivalence is denoted by $E^*(X)$. As numerical equivalence is the coarsest nontrivial equivalence relation on cycles (see \cite[p. 228]{Jannsen:2000aa}), we have a natural quotient map $E^*(X) \rightarrow N^*(X)$.

We have a refined notion of ambience:

\begin{definition} \label{eambient}
Let $X \subset V$ be a clean embedding, and denote by $i$ the inclusion. We will say that a surface $S$ in $X$ is $E$-\textit{ambient} if the cycle class $[S]$ of $S$ up to $E$-equivalence is in the image of the pullback homomorphism $i^* \colon E^2(V) \rightarrow E^2(X)$.
\end{definition}

The function $Q_{X \subset V}$ pulls back to cycle classes up to $E$-equivalence, and as $E$-ambience is finer than ambience, the key result of Proposition \ref{core} clearly continues to hold.

We extend the definition of decency:

\begin{definition} \label{epair}
We will call a clean embedding $X \subset V$ an $E$-\textit{decent pair} if $X \subset V$ is a decent pair and, in addition, the quotient map $E^2(V) \rightarrow N^2(V)$ has finite kernel.
\end{definition}

Thus $X \subset V$ is an $E$-decent pair if and only if both of the following conditions are satisfied:
\begin{enumerate}
\item The integral quadratic form $\alpha \mapsto \text{deg} ([X] \cdot \alpha \cdot \alpha)$ on $N^2(V)$ is positive definite.
\item The quotient map $E^2(V) \rightarrow N^2(V)$ has finite kernel.
\end{enumerate}

A clean embedding in which $E$-decency fails is given in Example \ref{nondecent} below.

We record analogues of Theorem \ref{bound} as well as of Corollaries \ref{numnongeneral}, \ref{numrational}, and \ref{numeuler}.

\begin{theorem} \label{ebound}
Let $E$ be an adequate equivalence relation, and let $X \subset V$ be an $E$-decent pair. Then for any $s \in \mathbb{Z}$, at most finitely many $E$-equivalence classes in $E^2(X)$ are representable by a smooth $E$-ambient surface $S$ in $X$ satisfying $\emph{deg} \left( c_1^2(\mathcal{T}_S) - c_2(\mathcal{T}_S) \right) \leq s$.
\end{theorem}
\begin{proof}
Any element $\alpha \in E^2(V)$ for which there exists a smooth surface $S \subset X$ with $\text{deg} \left( c_1^2(\mathcal{T}_S) - c_2(\mathcal{T}_S) \right) \leq s$ and $[S] = i^*(\alpha)$ satisfies $Q_{X \subset V}(\alpha) \leq s$. These $\alpha$ map in a finite-to-one manner under $E^2(V) \rightarrow N^2(V)$ to the group of cycle classes up to numerical equivalence, where, as in the proof of Theorem \ref{bound}, the set of their images is again shown to be finite.
\end{proof}

Applying Theorem \ref{ebound} in special cases, we see that:

\begin{corollary} \label{enongeneral}
At most finitely many $E$-equivalence classes in $E^2(X)$ are representable by a smooth $E$-ambient surface $S$ in $X$ not of general type.
\end{corollary}

\begin{corollary} \label{erational}
At most finitely many $E$-equivalence classes in $E^2(X)$ are representable by a smooth rational $E$-ambient surface $S$ in $X$.
\end{corollary}

\begin{corollary} \label{eeuler}
For any $r \in \mathbb{Z}$, at most finitely many $E$-equivalence classes in $E^2(X)$ are representable by a smooth $E$-ambient surface $S$ in $X$ satisfying $\chi(S, \mathcal{O}_S) \leq r$.
\end{corollary}

\subsection{Background in the theory of codimension-2 cycles}

We summarize important results and conjectures in the theory of algebraic cycles.

We first recall various adequate equivalence relations. We recall the Chow ring of algebraic cycles up to rational equivalence, introduced for example in \cite[\S 1.2]{Eisenbud:2016aa} and denoted by $CH^*(X) = \bigoplus_{k = 0}^{\text{dim}X} CH^k(X)$. We denote by $CH^k(X)_{\text{alg}}$ the subgroup of $CH^k(X)$ consisting of those codimension-$k$ cycle classes in $X$ which are algebraically equivalent to zero, defined for example as in \cite[Def. 10.3]{Fulton:1984aa}, and by $A^k(X)$ the group of codimension-$k$ cycles in $X$ up to algebraic equivalence. We denote by $CH^k(X)_{\text{hom}}$ the kernel of the cycle class map $\text{cl} \colon CH^k(X) \rightarrow H^{2k}(X, \mathbb{Z})$, defined for example as in Voisin's text \cite[\S 2.1.4]{Voisin:2014aa}, and by $B^k(X)$ the group of codimension-$k$ cycles in $X$ up to homological equivalence. We denote by $Gr^k(X)$ the Griffiths group $CH^k(X)_{\text{hom}} / CH^k(X)_{\text{alg}}$.

We denote by $CH^*(X)_{\mathbb{Q}}$ the Chow group with rational coefficients of a smooth variety $X$. We have:

\begin{conjecture}[{Bloch--Beilinson filtration (see \cite[Conj. 2.19]{Voisin:2014aa}, \cite[0.1]{Voisin:2016aa})}] \label{blochbeilinson}
To each smooth variety $X$ we may associate a filtration $F$ on each Chow group $CH^k(X)_{\mathbb{Q}}$ of $X$, satisfying the following properties:
\begin{enumerate}
\item $F^0CH^k(X)_{\mathbb{Q}} = CH^k(X)_{\mathbb{Q}}$ and $F^1CH^k(X)_{\mathbb{Q}} = CH^k(X)_{\emph{hom}, \mathbb{Q}}$.
\item $F$ respects the action of algebraic correspondences (see \cite[Def. 2.9]{Voisin:2014aa}), in the sense that for any correspondence $\Gamma \in CH^{\emph{dim}(X) + k}(X \times Y)$, $\Gamma_*(F^{\nu}CH^l(X)_{\mathbb{Q}}) \subset F^{\nu}CH^{l + k}(Y)_{\mathbb{Q}}$.
\item The induced map $\Gamma_* \colon Gr^{\nu}CH^l(X)_{\mathbb{Q}} \rightarrow Gr^{\nu}CH^{l + k}(Y)_{\mathbb{Q}}$ vanishes if the K\"{u}nneth component $[\Gamma]_* \colon H^{2l - \nu}(X, \mathbb{Q}) \rightarrow H^{2l - \nu + 2k}(Y, \mathbb{Q})$ of $[\Gamma] \in H^{2\emph{dim}(X) + 2k}(X \times Y, \mathbb{Q})$ vanishes.
\item $F^{k+1}CH^k(X)_{\mathbb{Q}} = 0$.
\end{enumerate}
\end{conjecture}

We recall also that for any smooth variety $X$ with hyperplane class $l \in H^2(X, \mathbb{Q})$, the Hard Lefschetz theorem (see \cite[(2.6)]{Voisin:2014aa}) gives, for any $k$, an isomorphism of Hodge structures
\begin{equation*}l^{\text{dim}(X) - k} \cup \colon H^k(X, \mathbb{Q}) \rightarrow H^{2\text{dim}(X) - k}(X, \mathbb{Q}),\end{equation*}
and thus an inverse isomorphism
\begin{equation*}(l^{\text{dim}(X) - k})^{-1} \cup \colon H^{\text{2dim}(X) - k}(X, \mathbb{Q}) \rightarrow H^k(X, \mathbb{Q}).\end{equation*}
This latter map defines a Hodge class $\lambda_{\text{dim}(X) - k}$ in $H^{2k}(X \times X, \mathbb{Q})$ (see \cite[Lem. 2.26]{Voisin:2014aa}).

\begin{conjecture}[{Lefschetz standard conjecture $B(X)$ (see \cite[Conj. 2.28]{Voisin:2014aa})}] \label{lefschetz}
Each Hodge class $\lambda_{\emph{dim}(X) - k} \in H^{2k}(X \times X, \mathbb{Q})$ is the class of an algebraic cycle with rational coefficients on $X \times X$.
\end{conjecture}

This is a special case of the following:

\begin{conjecture}[{Generalized Hodge Conjecture (see \cite[Conj. 2.40]{Voisin:2014aa})}] \label{ghc}
Consider a smooth variety $X$. Let $L \subset H^{2k}(X, \mathbb{Q})$ be a sub-Hodge structure of coniveau $\geq c$. Then there exists a closed algebraic subset $Y \subset X$ of codimension $c$ such that $L \subset \emph{ker} \left( H^{2k}(X, \mathbb{Q}) \rightarrow H^{2k}(X \backslash Y, \mathbb{Q}) \right)$.
\end{conjecture}

We refer to papers of Lewis \cite{Lewis:1989aa} and Murre \cite{Murre:1994aa} in what follows.

\begin{definition}[{Mumford (see \cite[p. 269]{Lewis:1989aa})}] \label{finitedimensional}
We will say that the group $CH^k(X)_{\text{alg}}$ is \textit{finite-dimensional} if there exists a smooth curve $E$ and a cycle class $\Gamma \in CH^k(E \times X)$ such that the induced map $\Gamma_* \colon CH_0(E)_{\text{alg}} \rightarrow CH^k(X)_{\text{alg}}$ is surjective.
\end{definition}

\begin{definition}[{Lieberman, Murre (see \cite[\S 6.6, II.]{Murre:1994aa})}]
Denoting by $AJ \colon CH^k(X)_{\text{hom}} \rightarrow J^k(X)$ the Abel--Jacobi map into the Griffiths intermediate Jacobian $J^k(X)$, the image in $J^k(X)$ of the restriction to $CH^2(X)_{\text{alg}}$ of $AJ$ is an abelian variety; we will denote it by $J_a^k(X)$.
\end{definition}

Regarding cycles of codimension 2, we have:

\begin{theorem}[{Murre \cite[\S 6.7, Thm. I]{Murre:1994aa}}] \label{universal}
The restricted Abel--Jacobi map $AJ \colon CH^2(X)_{\emph{alg}} \rightarrow J_a^2(X)$ is a universal regular homomorphism in the sense of \cite[\S 6.6]{Murre:1994aa}, and the right-hand side satisfies $\emph{dim} (J_a^2(X)) \leq \frac{1}{2} \emph{dim}(H^3(X, \mathbb{Q}))$.
\end{theorem}

\begin{theorem}[{Murre (see \cite[p. 268]{Lewis:1989aa})}] \label{isomorphism}
Suppose that $CH^2(X)_{\emph{alg}}$ is finite-dimensional. Then the homomorphism $AJ \colon CH^2(X)_{\emph{alg}} \rightarrow J_a^2(X)$ of Theorem \ref{universal} is an isomorphism.
\end{theorem}

These conjectures control the behavior of algebraic cycles on a smooth variety $X$:

\begin{conjecture}[{Bloch \cite[p. 2]{Bloch:1976aa}, Lewis \cite[p. 268]{Lewis:1989aa}}] \label{lewis}
Suppose that $H^{2,0}(X) = 0$, and assume that Conjecture \ref{ghc} holds for $X$. Then $CH^2(X)_{\emph{alg}}$ is finite-dimensional, and in particular satisfies $CH^2(X)_{\emph{alg}} \cong J_a^2(X)$ by Theorem \ref{isomorphism} above.
\end{conjecture}

\begin{theorem}[{Jannsen \cite[(7)]{Jannsen:2000aa}}] \label{griffiths}
Assume Conjectures \ref{blochbeilinson} and \ref{lefschetz}. If $H^3(X, \mathbb{Q}) = N^1H^3(X, \mathbb{Q})$, where $N$ is the coniveau filtration on $H^3(X, \mathbb{Q})$, then the Griffiths group $Gr^2(X)$ is of torsion.
\end{theorem}

\subsection{Specialization to the case $X = V$}

When we take $X = V$ in the above, all cleanness and ambience conditions become vacuous.

We specialize the definitions of decency:

\begin{definition}
If $X \subset X$ satisfies the condition of Definition \ref{pair}, we will say that $X$ is \textit{decent}.
\end{definition}

\begin{definition}
If, for some adequate equivalence relation $E$, $X \subset X$ satisfies the conditions of Definition \ref{epair}, then we will say that $X$ is $E$-\textit{decent}.
\end{definition}

In particular, $X$ is $E$-decent if and only if the integral intersection pairing on $N^2(X)$ is positive definite and the quotient map $E^2(X) \rightarrow N^2(X)$ has finite kernel.

\begin{example}[A non-$CH$-decent fourfold] \label{nondecent}
In a generalization of Mumford's theorem (see \cite[Thm. 3.13]{Voisin:2014aa}), Laterveer demonstrates that for any smooth fourfold $X$ for which $H^{2,0}(X) \neq 0$ and the Lefschetz standard conjecture holds (this class includes the abelian and hyper-K\"{a}hler fourfolds, by Charles and Markman \cite{Charles:2013aa}), the group $\text{ker} \left( CH^2(X) \rightarrow N^2(X) \right)$ is not just infinite but actually of infinite rank.

Indeed, if $\text{ker} \left( CH^2(X) \rightarrow N^2(X) \right) \otimes \mathbb{Q}$ were finitely generated over $\mathbb{Q}$, then so too would be $CH^2(X)_{\mathbb{Q}}$, as $N^2(X)_{\mathbb{Q}}$ is finitely generated. Yet upon selecting a finite $\mathbb{Q}$-basis for $CH^2(X)_{\mathbb{Q}}$, we would find that representatives of these basis elements are supported on a closed algebraic subset $X' \subset X$ for which $CH_2(X')_{\mathbb{Q}} \rightarrow CH_2(X)_{\mathbb{Q}}$ is surjective. This would contradict the main result of Laterveer \cite[Thm. 3.1]{Laterveer:2015aa}. (I would like to thank Robert Laterveer for this argument.)
\end{example}

We have the following specialization of Theorem \ref{ebound}:

\begin{theorem} \label{tbound}
Let $X$ be $E$-decent. Then for any $s \in \mathbb{Z}$, at most finitely many $E$-equivalence classes in $E^2(X)$ are representable by a smooth surface $S$ in $X$ satisfying $\emph{deg} \left( c_1^2(\mathcal{T}_S) - c_2(\mathcal{T}_S) \right) \leq s$.
\end{theorem}
\begin{proof}
We apply Theorem \ref{ebound} to the clean embedding $X \subset X$.
\end{proof}

We also have analogous specializations of Corollaries \ref{enongeneral}, \ref{erational}, and \ref{eeuler}.

As an example, we record a general result in the flavor of that of Ciliberto and Di Gennaro:

\begin{theorem} \label{cdg}
Any smooth fourfold $X$ with vanishing first Chern class and Picard rank $\rho(X) = 1$ is decent.
\end{theorem}
\begin{proof}
We apply Proposition \ref{picdecent} to the clean embedding $X \subset X$.
\end{proof}

\subsection{Examples of $E$-decent fourfolds}

We exhibit examples, some conjectural, of smooth fourfolds which are $E$-decent over finer equivalence relations $E$. 

We first study maps $E^2(X) \rightarrow N^2(X)$. In what follows, we let $X$ be a smooth variety:

\begin{proposition} \label{homologically}
The quotient map $B^2(X) \rightarrow N^2(X)$ has finite kernel.
\end{proposition}
\begin{proof}
The natural quotient map $B^2(X) \rightarrow N^2(X)$ annihilates only the torsion subgroup of the finitely generated group $B^2(X)$ (see \cite[19.3.2. (iii)]{Fulton:1984aa}), which is necessarily finite.
\end{proof}

\begin{proposition} \label{algebraically}
Suppose that $X$ satisfies $H^3(X, \mathbb{Q}) = 0$. Assume Conjectures \ref{blochbeilinson} and \ref{lefschetz}. Then the quotient map $A^2(X) \rightarrow B^2(X)$ is an isomorphism.
\end{proposition}
\begin{proof}
The work of Bloch and Ogus \cite[(7.5)]{Bloch:1974aa} gives an exact sequence
\begin{equation*}H^3(X, \mathbb{Z}) \rightarrow \Gamma(X, \mathscr{H}^3) \rightarrow A^2(X) \rightarrow B^2(X),\end{equation*}
where the sheaf $\mathscr{H}^3$ is defined in \cite[(7.4)]{Bloch:1974aa}. In fact, $\mathscr{H}^3$ is torsion free by results of Bloch and Srinivas \cite[p. 1240]{Bloch:1983aa} (see also \cite[Thm. 6.15]{Voisin:2014aa}). The assumption $H^3(X, \mathbb{Q}) = 0$ thus ensures that the first map is the zero map, and that the torsion free group $\Gamma(X, \mathscr{H}^3)$ injects into $A^2(X)$. Its image, the kernel of the third map, is precisely the Griffiths group $Gr^2(X)$, which by Theorem \ref{griffiths} we may assume to be of torsion. Thus the torsion free group $\Gamma(X, \mathscr{H}^3)$ is isomorphic to the torsion group $Gr^2(X)$, and we conclude that both are zero.
\end{proof}

\begin{proposition} \label{rationally}
Suppose that $X$ satisfies $H^3(X, \mathbb{Q}) = 0$ and $H^{2,0}(X) = 0$. Assume that Conjectures \ref{ghc} and \ref{lewis} hold for $X$. Then the quotient map $CH^2(X) \rightarrow A^2(X)$ is an isomorphism.
\end{proposition}
\begin{proof}
By Theorems \ref{universal} and \ref{isomorphism}, Conjecture \ref{lewis} predicts that $CH^2(X)_{\text{alg}}$, the kernel of the quotient map $CH^2(X) \rightarrow A^2(X)$, is isomorphic to the trivial group $J_a^2(X)$.
\end{proof}

Recalling Theorem \ref{cdg} and applying the results of this chapter, we now have:

\begin{theorem}
Let $X$ be a smooth fourfold with vanishing first Chern class and $\rho(X) = 1$. Then $X$ is homologically decent.
\end{theorem}
\begin{proof}
We apply Theorem \ref{cdg} and Proposition \ref{homologically}.
\end{proof}

Following Beauville \cite{Beauville:1983aa}, we have:

\begin{theorem}
Let $X$ be a smooth hyper-K\"{a}hler fourfold deformation equivalent to $S^{[2]}$ (see \cite[\S 6]{Beauville:1983aa}) for a K3 surface $S$, and suppose that $\rho(X) = 1$. Assume Conjectures \ref{blochbeilinson} and \ref{lefschetz}. Then $X$ is algebraically decent.
\end{theorem}
\begin{proof}
\cite[p. 779, note 4]{Beauville:1983aa} gives that $H^3(X, \mathbb{Q}) = 0$. We apply Theorem \ref{cdg} together with Propositions \ref{homologically} and \ref{algebraically}.
\end{proof}

We also have:

\begin{theorem} \label{ci}
Let $X$ be a smooth complete intersection Calabi--Yau fourfold. Assume Conjectures \ref{blochbeilinson} and \ref{lefschetz}, and assume that Conjectures \ref{ghc} and \ref{lewis} hold for $X$. Then $X$ is rationally decent.
\end{theorem}
\begin{proof}
The Lefschetz hyperplane theorem implies that $X$ satisfies $\rho(X) = 1$ as well as $H^3(X, \mathbb{Q}) = 0$ and $H^{2,0}(X) = 0$. Theorem \ref{cdg} together with Propositions \ref{homologically}, \ref{algebraically}, and \ref{rationally} complete the proof.
\end{proof}

The topologically distinct such fourfolds are given as configuration matrices in say Gray, Haupt, and Lukas \cite{Gray:2013aa}. In particular, the smooth sextic fourfold $X$ in $\mathbb{P}^5$ is conjecturally rationally decent.

\section{Detailed case studies of decent pairs and lattice point counting} \label{detailed}

We demonstrate the operation of the theory on example pairs $X \subset V$ for which the associated function $Q_{X \subset V}$ can be determined exactly. $V$ will be a projective space, a Grassmannian, or a product of projective spaces in the examples that follow.

The examples $X \subset V$ that follow all feature isomorphisms $CH^2(V) \rightarrow N^2(V)$, so that the relevant results hold over any adequate equivalence relation. For simplicity, we present results only over numerical equivalence.

Decency in each case is computed by hand. The theoretical result of Proposition \ref{picdecent} preempts this computation in the cases below in which $V$ is a projective space; in the cases treated below in which $V$ is a Grassmannian or a product of projective spaces, only this result's conjectural generalization in Remark \ref{picdecentgen} applies (though its predictions obtain in each case). Recall that for each of these varieties $V$ the Hodge conjecture is known, so that the result of Remark \ref{picdecentopp} applies.

The question of smooth rational surfaces which arise in a smooth fourfold $X$ with vanishing first Chern class as a generically transverse intersection in an ambient variety $V$ is of course trivial whenever $V$ is a homogeneous variety (admitting the transitive action of a group). Indeed, in any such case, given any such smooth rational surface the ambient intersecting subvariety could sweep out across the whole fourfold, covering it with rational curves, contradicting the fact that a smooth fourfold with vanishing first Chern class can never be uniruled (see Koll\'{a}r \cite[IV, Cor. 1.11]{Kollar:1999aa}). (I would like to thank Vyacheslav Shokurov for this argument.)

The technique of Theorem \ref{bound}, on the other hand, serves to control the behavior of quite general sorts of smooth surfaces. We apply this theorem in what follows.

\subsection{Projective spaces}

We first take examples in which $V$ is a projective space. Recall that $CH^*(\mathbb{P}^n) \cong \mathbb{Z}[\zeta]/(\zeta^{n+1})$, where $\zeta$ is the class of a hyperplane (see \cite[Thm. 2.1]{Eisenbud:2016aa}), so that $CH^2(\mathbb{P}^n) \cong \mathbb{Z}$, and furthermore that the natural map $CH^*(\mathbb{P}^n) \rightarrow N^*(\mathbb{P}^n)$ is an isomorphism (see \cite[Ex. 19.1.11]{Fulton:1984aa}).

\begin{example}[Smooth complete intersection fourfolds]
While we have treated smooth complete intersection fourfolds abstractly in Theorem \ref{ci} above, we now consider such fourfolds $X$ as members of decent pairs $X \subset \mathbb{P}^n$.

Any smooth complete intersection $X$ of hypersurfaces of degrees $a_1, \ldots , a_k$ in $\mathbb{P}^{4 + k}$ with $\sum_{i = 1}^k a_i = 5 + k$ is a Calabi--Yau fourfold, by the equality $K_{\mathbb{P}^{4 + k}} \cong \mathcal{O}_{\mathbb{P}^{4 + k}}(-5 - k)$ and the adjunction formula. $X$ is cleanly embedded in $\mathbb{P}^{4 + k}$ in this case by the condition \ref{clean2}. following Definition \ref{clean} above. In particular, by Lemma \ref{restriction}, $c_2(\mathcal{T}_X)$ is the restriction to $X$ of a cycle class on $\mathbb{P}^{4 + k}$. We now compute this class.

We have the Euler sequence on $\mathbb{P}^{4 + k}$ and the normal bundle sequence on $X$:
\begin{gather*}
\textit{Euler sequence: }0\rightarrow \mathcal{O}_{\mathbb{P}^{4 + k}} \rightarrow \left( \mathcal{O}_{\mathbb{P}^{4 + k}}(1) \right)^{\oplus 5 + k} \rightarrow \mathcal{T}_{\mathbb{P}^{4 + k}} \rightarrow 0,\\
\textit{Normal bundle sequence: }\textstyle 0 \rightarrow \mathcal{T}_X \rightarrow \left. \mathcal{T}_{\mathbb{P}^{4 + k}} \right|_X \rightarrow \left. \bigoplus_{i = 1}^k \mathcal{O}_{\mathbb{P}^{4 + k}}(a_i) \right|_X \rightarrow 0.
\end{gather*}
These give:
\begin{align*}
c_2 \left( \mathcal{T}_{\mathbb{P}^{4 + k}} \right) &= c_2 \left( \left( \mathcal{O}_{\mathbb{P}^{4 + k}}(1) \right)^{\oplus 5 + k} \right) - c_1(\mathcal{O}_{\mathbb{P}^{4 + k}}) \cdot c_1(\mathcal{T}_{\mathbb{P}^{4 + k}}) - c_2(\mathcal{O}_{\mathbb{P}^{4 + k}}) \\
&= \textstyle \binom{5 + k}{2} c_1^2(\mathcal{O}_{\mathbb{P}^{4 + k}}(1)), \\
c_2(\mathcal{T}_X) &= \textstyle c_2 \left( \left. \mathcal{T}_{\mathbb{P}^{4 + k}} \right|_X \right) - c_1(\mathcal{T}_X) \cdot c_1 \left( \left. \bigoplus_{i = 1}^k \mathcal{O}_{\mathbb{P}^{4 + k}} (a_i) \right|_X \right) - c_2 \left( \left. \bigoplus_{i = 1}^k \mathcal{O}_{\mathbb{P}^{4 + k}}(a_i) \right|_X \right) \\
&= \textstyle \binom{5 + k}{2} c_1^2(\mathcal{O}_X(1)) - \left( \sum_{i < j}^k a_ia_j \cdot c_1^2(\mathcal{O}_X(1)) \right) \\
&= \textstyle \left( \binom{5 + k}{2} - \sum_{i < j}^k a_ia_j \right) \zeta^2.
\end{align*}
Finally, $[X] = (\prod_{i=1}^k a_i) \cdot \zeta^k \in N_4(\mathbb{P}^{4 + k})$.

Identifying $N^2(\mathbb{P}^{4 + k}) \cong \mathbb{Z}$ via $\zeta^2$, the quadratic form $\alpha \mapsto \text{deg} ([X] \cdot \alpha \cdot \alpha)$ identifies to $x_1 \mapsto (\prod_{i=1}^k a_i) \cdot x_1^2$, which is clearly positive definite. Thus $X \subset \mathbb{P}^{k + 4}$ is a decent pair, as predicted by Proposition \ref{picdecent}.

Lemma \ref{quadratic} gives the following formula for the associated function $Q_{X \subset \mathbb{P}^{4 + k}}$ on $N^2(\mathbb{P}^{4 + k}) \cong \mathbb{Z}$:
\begin{equation*}\textstyle Q_{X \subset \mathbb{P}^{4 + k}}(x_1) = \left( \prod_{i=1}^k a_i \right) \cdot x_1^2 - \left( (\prod_{i=1}^k a_i) \left( \binom{5 + k}{2} - \sum_{i < j}^k a_ia_j \right) \right) \cdot x_1.\end{equation*}
\end{example}

\begin{example}[The smooth sextic fourfold]
Taking $k = 1, a_1 = 6$ in the above gives the sextic fourfold $X$ in $\mathbb{P}^5$. We have the following associated function on $N^2(\mathbb{P}^5) \cong \mathbb{Z}$:
\begin{equation*}Q_{X \subset \mathbb{P}^5}(x_1) = 6x_1^2 - 90x_1.\end{equation*}

We have the following specialization of Theorem \ref{bound}:

\begin{theorem} \label{sextic}
Consider the smooth sextic fourfold $X \subset \mathbb{P}^5$. Let $s \in \mathbb{Z}$. Then at most
\begin{equation*}\frac{\sqrt{90^2 + 24s}}{6} + 1\end{equation*}
elements of $N^2(X)$ are representable by a smooth ambient surface $S$ in $X$ satisfying $\emph{deg} \left( c_1^2(\mathcal{T}_S) - c_2(\mathcal{T}_S) \right) \leq s$. (If $90^2 + 24s$ is negative, then $N^2(X)$ has no such elements.)
\end{theorem}
\begin{proof}
If $s$ is such that the discriminant $90^2 + 24s < 0$, then $Q_{X \subset \mathbb{P}^5}(x_1) \leq s$ has no solutions and we apply Proposition \ref{core}. Otherwise, the quadratic formula shows that $Q_{X \subset \mathbb{P}^5}(x_1) = s$ has real solutions which are separated by a real interval of width $\frac{\sqrt{90^2 + 24s}}{6}$. The number of lattice points in such an interval is at most $\left[\frac{\sqrt{90^2 + 24s}}{6}\right] + 1$, and we apply Proposition \ref{core} again.
\end{proof}
\end{example}

\subsection{Grassmannians}

We discuss examples in which the ambient variety is a Grassmannian. We follow the treatment of Grassmannians given in \cite[\S 4]{Eisenbud:2016aa}. Recall that the Grassmannian $G(l,V)$ denotes the set of $l$-dimensional subspaces $\Lambda$ of an $n$-dimensional vector space $V$ over $\mathbb{C}$, given the structure of a smooth $l(n-l)$-dimensional variety via the Pl\"{u}cker embedding $G(l,V) \hookrightarrow \mathbb{P}(\wedge^l(V)) = \mathbb{P}^{\binom{n}{l} - 1}$. We have tautological exact sequences
\begin{gather*}
0 \rightarrow \mathscr{I}_l \rightarrow V \otimes \mathcal{O}_{G(l,V)} \rightarrow Q \rightarrow 0, \\
0 \rightarrow Q^* \rightarrow V^* \otimes \mathcal{O}_{G(l,V)} \rightarrow \mathscr{E}_l \rightarrow 0
\end{gather*}
on $G(l,V)$, where in particular $\mathscr{I}_l$ is the tautological subbundle on $G(l,V)$ and $\mathscr{E}_l$ is its dual. Fixing a full flag $0 \subset V_1 \subset \cdots \subset V_{n - 1} \subset V_n = V$, for any decreasing sequence $n - l \geq a_1 \geq \cdots \geq a_l \geq 0$ of integers we denote by $\Sigma_{a_1, \ldots , a_l}$ the Schubert cycle consisting of those subspaces $\Lambda \subset V$ such that $\text{dim}(\Lambda \cap V_{n - l + k - a_i}) \geq k$ for each $k = 1, \ldots , l$. We occasionally suppress some or all trailing zeros in the indices $(a_1, \ldots , a_l)$. The subvariety $\Sigma_{a_1, \ldots , a_l}$ has codimension $a_1 + \cdots + a_l$ in $G(l,V)$. The Chow group $CH^k(G(l,V))$ is generated freely over $\mathbb{Z}$ by the Schubert classes $[\Sigma_{a_1, \ldots , a_l}$] of the Schubert cycles $\Sigma_{a_1, \ldots , a_l}$ of codimension $k$ in $G(l,V)$ (see \cite[Cor. 4.7]{Eisenbud:2016aa}). The multiplicative structure in the Chow ring $CH^*(G(l,V))$ is given by the Littlewood--Richardson rule (see \cite[Cor. 4.8]{Eisenbud:2016aa}). Once again in this case the natural map $CH^*(G(l,V)) \rightarrow N^*(G(l,V))$ is an isomorphism (see \cite[Ex. 19.1.11]{Fulton:1984aa}).

\subsubsection{The Fano variety of lines in a general cubic fourfold}

The Fano variety $F$ of lines contained in a general cubic fourfold (see Beauville and Donagi \cite{Beauville:1985aa}, Voisin \cite{Voisin:2008aa}), a smooth fourfold with vanishing first Chern class, is constructed in the Grassmannian $G(2,V)$, where $V$ is a $6$-dimensional vector space, as the zero locus of a general section of the bundle $\text{Sym}^3(\mathscr{E}_2)$ on $G(2,V)$. In particular, $F$ is cleanly embedded in $G(2,V)$.

We compute the class in $N^2(G(2,V))$ which restricts to $c_2(\mathcal{T}_F)$ on $F$. In what follows, we abbreviate $c_k := c_k(\mathscr{E}_2)$. Tensoring the tautological exact sequence on $G(2,V)$ with $\mathscr{E}_2$, we get an analogue of the Euler sequence:
\begin{equation*}\textit{Euler sequence: }0 \rightarrow \mathscr{I}_2 \otimes \mathscr{E}_2 \rightarrow (\mathscr{E}_2)^{\oplus 6} \rightarrow \mathcal{T}_{G(2,V)} \rightarrow 0.\end{equation*}
We also have the normal bundle exact sequence on $F$:
\begin{equation*}\textit{Normal bundle sequence: }0 \rightarrow \mathcal{T}_F \rightarrow \left. \mathcal{T}_{G(2,V)} \right|_F \rightarrow \left. \text{Sym}^3(\mathscr{E}_2) \right|_F \rightarrow 0.\end{equation*}
Using symmetric polynomials in the Chern roots of $\mathscr{E}_2$, we compute the Chern class expressions:
\begin{equation*}c_2 \left( \text{Sym}^3(\mathscr{E}_2) \right) = 11c_1^2 + 10c_2, \quad c_4 \left( \text{Sym}^3(\mathscr{E}_2) \right) = 18c_1^2c_2 + 9 c_2^2.\end{equation*}
These together give:
\begin{align*}
c_2 \left( \mathcal{T}_{G(2,V)} \right) &= c_2 \left( (\mathscr{E}_2)^{\oplus 6} \right) - c_1(\mathscr{I}_2 \otimes \mathscr{E}_2) \cdot c_1(\mathcal{T}_{G(2,V)}) - c_2(\mathscr{I}_2 \otimes \mathscr{E}_2) \\
&= (15c_1^2 + 6c_2) - (-c_1^2 + 4c_2) \\
&= 16c_1^2 + 2c_2, \\
c_2(\mathcal{T}_F) &= c_2 \left( \left. \mathcal{T}_{G(2,V)} \right|_F \right) - c_1(\mathcal{T}_F) \cdot c_1 \left( \left. \text{Sym}^3(\mathscr{E}_2) \right|_F \right) - c_2 \left( \left. \text{Sym}^3(\mathscr{E}_2) \right|_F \right) \\
&= (16c_1^2 + 2c_2) - (11c_1^2 + 10c_2) \\
&= 5c_1^2 - 8c_2.
\end{align*}
Finally, the above expression for $c_4 \left( \text{Sym}^3(\mathscr{E}_2) \right)$ and the Littlewood--Richardson rule (see \cite[Prop. 4.11]{Eisenbud:2016aa} for its specialization to the case $l = 2$) give that $[F] = 18 \cdot \Sigma_{3,1} + 27 \cdot \Sigma_{2,2} \in N_4(G(2,V))$.

Identifying $N^2(G(2,V)) \cong \mathbb{Z}^2$ via $e_1 = \Sigma_{2,0}, e_2 = \Sigma_{1,1}$, the quadratic form $\alpha \mapsto \text{deg} ([F] \cdot \alpha \cdot \alpha)$ identifies to the form with Gram matrix
\begin{equation*}
\left[\begin{array}{cc} 
([F] \cdot e_1 \cdot e_1) & ([F] \cdot e_1 \cdot e_2) \\
([F] \cdot e_2 \cdot e_1) & ([F] \cdot e_2 \cdot e_2)
\end{array}\right]
=
\left[\begin{array}{cc}
45 & 18 \\
18 & 27
\end{array}\right].
\end{equation*}
This matrix is positive definite, and so $F \subset G(2,V)$ is a decent pair.

We have the associated function $Q_{F \subset G(2,V)}$ on $N^2(G(2,V)) \cong \mathbb{Z}^2$:
\begin{equation*}Q_{F \subset G(2,V)}(x_1, x_2) = 45x_1^2 + 36x_1x_2 + 27x_2^2 - 171x_1 - 9x_2.\end{equation*}

\begin{theorem} \label{fano}
Consider the Fano variety $F \subset G(2,V)$ of lines in a general cubic fourfold. Let $s \in \mathbb{Z}$. Then at most
\begin{equation*}\frac{\pi(s + 207)}{\sqrt{891}} + 8 + 8 \cdot 2 \sqrt{\frac{s + 207}{9(4 - \sqrt{5})}}\end{equation*}
elements of $N^2(F)$ are representable by a smooth ambient surface $S$ in $F$ satisfying $\emph{deg} \left( c_1^2(\mathcal{T}_S) - c_2(\mathcal{T}_S) \right) \leq s$. (If $s + 207$ is negative, then $N^2(F)$ has no such elements.)
\end{theorem}
\begin{proof}
Beginning with the quadratic equation $Q_{F \subset G(2,V)}(x_1, x_2) = s$, we compute a standard discriminant (see for example Lawrence \cite[p. 63]{Lawrence:2014aa}):
\begin{equation*}
-27
\left|\begin{array}{ccc} 
45 & 18 & \sfrac{-171}{2} \\
18 & 27 & \sfrac{-9}{2} \\
\sfrac{-171}{2} & \sfrac{-9}{2} & -s
\end{array}\right|
= 27(891s + 184437) = 24057(s + 207).
\end{equation*}
If $s$ is such that this quantity is negative, then $Q_{F \subset G(2,V)}(x_1, x_2) \leq s$ has no solutions and we apply Proposition \ref{core}. Otherwise, $Q_{F \subset G(2,V)}(x_1, x_2) \leq s$ describes a real ellipse. We count its lattice points. Applying the translation $(u_1, u_2) \mapsto \left( u_1 + \frac{5}{2}, u_2 - \frac{3}{2} \right) = (x_1, x_2)$, the equation $Q_{F \subset G(2,V)}(x_1, x_2) \leq s$ pulls back to the ellipse
\begin{equation*}45u_1^2 + 36u_1u_2 + 27u_2^2 \leq s + 207.\end{equation*}
The following estimate depends only on the ellipse's area, and so it suffices to count this ellipse's lattice points instead. By a standard calculation, its area is $\frac{\pi(s + 207)}{\sqrt{45 \cdot 27 - 18^2}}$. The length of its major axis, and thus also its width, is $2\sqrt{\frac{s + 207}{9(4 - \sqrt{5})}}$, as can be seen most easily by orthogonally diagonalizing its associated quadratic form ($9(4 - \sqrt{5})$ is its matrix's smallest eigenvalue). A na\"{i}ve counting method (see Cohn \cite[p. 161]{Cohn:1980aa}) now yields the error term given above.
\end{proof}

\begin{example}[The Fano surface of lines in a hyperplane section of the fourfold]
We consider in particular the Fano surface $S$ of lines contained in a hyperplane section $H \cap F$ of $F$. Voisin shows that certain singular such $S$ are rational (see the proof of \cite[Lem. 3.2]{Voisin:2008aa}) and that in fact the rational $S$ belong to infinitely many rational equivalence classes (this follows from the existence of the rational self-map of \cite[Thm. 2]{Voisin:2004aa}). On the other hand, the techniques of our paper serve to recover that any smooth $S$ cannot be rational. Indeed, any Fano surface $S$ is given as the zero locus of expected codimension in $F$ of a section of the restriction to $F$ of $\mathscr{E}_2$, and in particular is ambient, by the criterion \ref{ambient2}. following Definition \ref{ambient}. Because $c_2(\mathscr{E}_2) = (0,1)$ in $N^2(G(2,V)) \cong \mathbb{Z}^2$, when $S$ is in addition smooth the inequality $Q_{F \subset G(2,V)}(0,1) = 18 > 6$ ensures the irrationality of $S$.

In fact, for smooth $S$ we can independently compute both sides of Proposition \ref{core}. The results of Harris and Tu \cite{Harris:1984aa} permit the calculation of the Chern numbers of any subvariety defined as a smooth degeneracy locus of expected dimension. Applying the case $\text{dim}(Z) = 2$ of \cite[p. 474]{Harris:1984aa} to the smooth Fano surface of lines $S$ as above, we compute:
\begin{equation*}\text{deg} \left( c_1^2(\mathcal{T}_S) \right) = 45, \quad \text{deg} \left( c_2(\mathcal{T}_S) \right) = 27.\end{equation*}
In particular, this surface fails to be rational, and the difference between its Chern numbers is 18, as predicted by Proposition \ref{core}.
\end{example}

\subsubsection{The Debarre--Voisin fourfolds}

We study the fourfolds introduced by Debarre and Voisin in \cite{Debarre:2010aa}. The smooth hyper-K\"{a}hler fourfolds $Y_\sigma$ vary in a 20-dimensional family, parameterized by the general 3-forms $\sigma \in \bigwedge^3 V^*$ on a 10-dimensional vector space $V$. Each $Y_\sigma$ is defined as the locus in $G(6,V)$ consisting of those 6-dimensional subspaces on which $\sigma$ vanishes identically, or, in other terms, as the zero locus in $G(6,V)$ of the section of $\bigwedge^3 \mathscr{E}_6$ determined by $\sigma$. Because $\sigma$ is general and $\bigwedge^3 \mathscr{E}_6$ is globally generated, $Y_\sigma$ is a smooth fourfold. That $Y_\sigma$ is also irreducible and Calabi--Yau, and in fact hyper-K\"{a}hler, is shown in \cite{Debarre:2010aa}. The embedding $Y_\sigma \subset G(6,V)$ is clean, by the condition \ref{clean3}. of Definition \ref{clean}.

We compute the class in $N^2(G(6,V))$ which restricts to $c_2 \left( \mathcal{T}_{Y_\sigma} \right)$ on $Y_\sigma$. We denote again $c_k := c_k(\mathscr{E}_6)$. We have the exact sequences:
\begin{gather*}
\textit{Euler sequence: }0 \rightarrow \mathscr{I}_6 \otimes \mathscr{E}_6 \rightarrow (\mathscr{E}_6)^{\oplus 10} \rightarrow \mathcal{T}_{G(6,V)} \rightarrow 0, \\
\textit{Normal bundle sequence: }\textstyle 0 \rightarrow \mathcal{T}_{Y_\sigma} \rightarrow \left. \mathcal{T}_{G(6,V)} \right|_{Y_\sigma} \rightarrow \left. \bigwedge^3(\mathscr{E}_6) \right|_{Y_\sigma} \rightarrow 0.
\end{gather*}
We also compute that $c_2 \left( \bigwedge^3(\mathscr{E}_6) \right) = 45c_1^2 + 6c_2$. These give:
\begin{align*}
c_2 \left( \mathcal{T}_{G(6,V)} \right) &= c_2 \left( (\mathscr{E}_6)^{\oplus 10} \right) - c_1(\mathscr{I}_6 \otimes \mathscr{E}_6) \cdot c_1(\mathcal{T}_{G(6,V)}) - c_2(\mathscr{I}_6 \otimes \mathscr{E}_6) \\
&= (45c_1^2 + 10c_2) - (-5c_1^2 + 12c_2) \\
&= 50c_1^2 - 2c_2, \\
c_2(\mathcal{T}_{Y_\sigma}) &= \textstyle c_2 \left( \left. \mathcal{T}_{G(6,V)} \right|_{Y_\sigma} \right) - c_1(\mathcal{T}_{Y_\sigma}) \cdot c_1 \left( \left. \bigwedge^3(\mathscr{E}_6) \right|_{Y_\sigma} \right) - c_2 \left( \left. \bigwedge^3(\mathscr{E}_6) \right|_{Y_\sigma} \right) \\
&= (50c_1^2 - 2c_2) - (45c_1^2 + 6c_2) \\
&= 5c_1^2 - 8c_2.
\end{align*}

\begin{remark} \label{tangent}
Interestingly, this expression appears also in the smaller Grassmannian treated above. We are currently unable to explain this.
\end{remark}

We have the following intersection numbers on $Y_\sigma$, given by \cite[Lem. 4.5]{Debarre:2010aa}:
\begin{equation*}c_1^4 = 1452, \quad c_1^2c_2 = 825, \quad c_2^2 = 477, \quad c_1c_3 = 330, \quad c_4 = 105.\end{equation*}
The quadratic form $\alpha \mapsto \text{deg} ([Y_\sigma] \cdot \alpha \cdot \alpha)$ is again positive definite, and $Y_\sigma \subset G(6,V)$ is decent. Identifying $N^2(G(6,V)) \cong \mathbb{Z}^2$ via $e_1 = \Sigma_{2,0}, e_2 = \Sigma_{1,1}$, we have the associated function:
\begin{equation*}Q_{Y_\sigma \subset G(6,V)}(x_1, x_2) = 279x_1^2 + 696x_1x_2 + 477x_2^2 - 351x_1 - 309x_2.\end{equation*}

\begin{theorem}
Consider the smooth Debarre--Voisin fourfold $Y_\sigma$. Let $s \in \mathbb{Z}$. Then at most
\begin{equation*}\frac{\pi(s + 207)}{\sqrt{11979}} + 8 + 8 \cdot 2 \sqrt{\frac{s + 207}{3(126 - \sqrt{14545})}}\end{equation*}
elements of $N^2(Y_\sigma)$ are representable by a smooth ambient surface $S \subset Y_\sigma$ satisfying $\emph{deg} \left( c_1^2(\mathcal{T}_S) - c_2(\mathcal{T}_S) \right) \leq s$. (If $s + 207$ is negative, then $N^2(Y_\sigma)$ has no such elements.)
\end{theorem}
\begin{proof}
As in the Fano variety above, we compute the discriminant of \cite[p. 63]{Lawrence:2014aa}:
\begin{equation*}
-477
\left|\begin{array}{ccc} 
279 & 348 & \sfrac{-351}{2} \\
348 & 477 & \sfrac{-309}{2} \\
\sfrac{-351}{2} & \sfrac{-309}{2} & -s
\end{array}\right|
= 477(11979s + 2479653) = 5713983(s + 207).
\end{equation*}
If $s$ is such that this quantity is negative, then $Q_{Y_\sigma \subset G(6,V)}(x_1, x_2) \leq s$ has no solutions and we apply Proposition \ref{core}. Otherwise, $Q_{Y_\sigma \subset G(6,V)}(x_1, x_2) \leq s$ describes a real ellipse. We count its lattice points. Translating again $(u_1, u_2) \mapsto \left( u_1 + \frac{5}{2}, u_2 - \frac{3}{2} \right) = (x_1, x_2)$, the equation $Q_{Y_\sigma \subset G(6,V)}(x_1, x_2) \leq s$ pulls back to the ellipse
\begin{equation*}279u_1^2 + 696u_1u_2 + 477u_2^2 \leq s + 207.\end{equation*}
We again count this ellipse's lattice points instead. This ellipse has area $\frac{\pi(s + 207)}{\sqrt{279 \cdot 477 - 348^2}}$ and width $2 \sqrt{\frac{s + 207}{3(126 - \sqrt{14545})}}$; the method of \cite[p. 161]{Cohn:1980aa} again gives the estimate above.
\end{proof}

\begin{remark}
That the two quadratic functions $Q_{F \subset G(2,V)}$ and $Q_{Y_\sigma \subset G(6,V)}$ attain their minima at the same point, namely $(x_1, x_2) = \left( \frac{5}{2}, -\frac{3}{2} \right)$, is easily explained, provided that the observation of Remark \ref{tangent} is assumed. Indeed, arguments of elementary linear algebra demonstrate that for any say real vector space $W$ equipped with a symmetric bilinear form $A$, the quadratic function $\alpha \mapsto A(\alpha \, ; \alpha) - A(\alpha \, ; \nu)$ on $W$ becomes homogeneous exactly around $\frac{\nu}{2}$ (independently of the form $A$). Taking now an embedding of any smooth $2d$-dimensional $X$ into a smooth variety $V$, we apply this fact to the symmetric bilinear form $(\alpha_1, \alpha_2) \mapsto \text{deg} \left( [X] \cdot \alpha_1 \cdot \alpha_2 \right)$ on $N^d(V)$, choosing here $\nu$ freely. We recall, of course, the common second Chern class expression $5c_1^2 - 8c_2 = 5 \cdot \Sigma_{2,0} - 3 \cdot \Sigma_{1,1} = \nu$.
\end{remark}

\begin{remark}
That the two quadratic functions $Q_{F \subset G(2,V)}$ and $Q_{Y_\sigma \subset G(6,V)}$ attain the same minimum value (namely $-207$) at this point also admits an explanation. Indeed, the value of the above quadratic function $\alpha \mapsto A(\alpha \, ; \alpha) - A(\alpha \, ; \nu)$ at its vertex $\frac{\nu}{2}$ is exactly $-\frac{1}{4}A(\nu \, ; \nu)$. Returning to the embedding $X \subset V$, if in this case in addition $i^*(\nu) = c_d \left( \mathcal{T}_X \right)$, then this value is in fact $-\frac{1}{4} c_d^2 \left( \mathcal{T}_X \right)$. The equality thus reflects exactly the equality of Chern numbers $c_2^2 \left( \mathcal{T}_F \right) = c_2^2 \left( \mathcal{T}_{Y_\sigma} \right) = 4 \cdot 207 = 828$. This equality in turn follows from the deformation equivalence of $F$ and $Y_\sigma$, a consequence of say Huybrechts \cite[\S 7, Thm. 1.1]{Huybrechts:2016aa} as well as \cite[Prop. 2]{Beauville:1985aa} and \cite[Thm. 4.1]{Debarre:2010aa}.
\end{remark}

\begin{example}[A non-ambient surface] \label{nonambient}
It is shown in \cite[Rem. 2.5]{Debarre:2010aa} that $\sigma$ can be chosen so that $Y_\sigma$, while remaining smooth and hyper-K\"{a}hler, contains a plane. This plane in $Y_\sigma$ arises from a plane in the Grassmannian which is contained in $Y_\sigma$, and not from, say, an ambient subvariety of codimension 2 in $G(6,V)$ which is generically transverse to $Y_\sigma$.

In fact, it can be proven that $\mathbb{P}^2 \subset Y_\sigma$ is not ambient in $G(6,V)$. If this plane were ambient, with $[\mathbb{P}^2] = i^*(\alpha)$ say, then Proposition \ref{euler} (using $r = 1$) and Proposition \ref{alternate} (using $r = 1, q = 9$) would demonstrate that $Q_{Y_\sigma \subset G(6,V)}(\alpha) = 6$. Yet the function $Q_{Y_\sigma \subset G(6,V)}$ never assumes the value $6$ in $N^2(G(6,V)) \cong \mathbb{Z}^2$.
\end{example}

\begin{example}[The surface of $6$-spaces which intersect a fixed $3$-space]
We study further the surfaces $S \subset Y_\sigma$ with $[S] = i^*(1,0)$ in $N^2(Y_\sigma)$. A family of such surfaces is given by the smooth intersections $S$ of the expected codimension of $Y_\sigma$ with the dependency locus in $G(6,V)$ of three sections of the tautological quotient bundle $Q$. Indeed, $c_2(Q) = \Sigma_{2,0} = (1,0)$ in $N^2(G(6,V))$.

The results of Harris and Tu \cite[p. 474]{Harris:1984aa}, again, determine directly the Chern numbers of such $S$:
\begin{equation*}\text{deg} \left( c_1^2(\mathcal{T}_S) \right) = 900, \quad \text{deg} \left( c_2(\mathcal{T}_S) \right) = 972.\end{equation*}
The difference $\text{deg} \left( c_1^2(\mathcal{T}_S) - c_2(\mathcal{T}_S) \right)$ is $Q_{Y_\sigma \subset G(6,V)}(1,0) = -72$, again as predicted by Proposition \ref{core}. These numbers also give the holomorphic Euler characteristic $\chi(S, \mathcal{O}_S) = 156$.

We demonstrate by a further method that the Chern numbers of $S$ are as given by \cite{Harris:1984aa} above. (I would like to thank Steven Sam for explaining this procedure.)

The form $\sigma$ defining $Y_\sigma$ gives a map of bundles $\mathcal{O}_{G(6,V)} \rightarrow \bigwedge^3(\mathscr{E}_6)$. Dualizing this map gives the Koszul complex $\mathbf{F}^\bullet$ as below, in which the degree of each term is indicated beneath it; because $Y_\sigma$ is of the expected codimension and $G(6,V)$ is locally Cohen--Macaulay, the Koszul complex $\mathbf{F}^\bullet$ on $G(6,V)$ is in fact exact (see for example \cite[Ex. 17.20]{Eisenbud:1995aa}):
\begin{equation*}
\begin{tabular}{*{14}{>{$}c<{$}@{\hskip 0.1in}}}
\mathbf{F}^\bullet \colon & 0 & \rightarrow & \bigwedge^{20} \bigwedge^3 (\mathscr{I}_6) & \rightarrow \cdots \rightarrow & \bigwedge^2 \bigwedge^3 (\mathscr{I}_6) & \rightarrow & \bigwedge^3 (\mathscr{I}_6) & \rightarrow & \mathcal{O}_{G(6,V)} & \rightarrow & \mathcal{O}_{Y_\sigma} & \rightarrow & 0. \\
& & & \scriptstyle{-20} & & \scriptstyle{-2} & & \scriptstyle{-1} & & \scriptstyle{0} & & &
\end{tabular}
\end{equation*}

Assuming now that a map $\mathcal{O}_{G(6,V)}^3 \rightarrow Q$ defines a dependency locus $S'$ of the expected codimension 2 in $G(6,V)$, we have an Eagon--Northcott complex (in this case, a Hilbert--Burch complex) resolving the structure sheaf $\mathcal{O}_{S'}$ of $S'$ (see Eisenbud \cite[\S A2H]{Eisenbud:2005aa}, Gruson, Sam, and Weyman \cite[\S 1.1, Thm. 7]{Gruson:2013aa}):
\begin{equation*}
\begin{tabular}{*{12}{>{$}c<{$}@{\hskip 0.1in}}}
\mathbf{HB}^\bullet \colon & 0& \rightarrow & \left( \bigwedge^4(Q^*) \right)^{\oplus 3} & \rightarrow & \bigwedge^{3}(Q^*) & \rightarrow & \mathcal{O}_{G(6,V)} & \rightarrow & \mathcal{O}_{S'} & \rightarrow & 0. \\
& & & \scriptstyle{-2} & & \scriptstyle{-1} & & \scriptstyle{0} & & & &
\end{tabular}
\end{equation*}
Finally, $Y_\sigma$ and $S'$, both degeneracy loci of the expected codimension, are Cohen--Macaulay, and furthermore they intersect in the expected dimension, so that the tensor product $\left( \mathbf{F} \otimes \mathbf{HB} \right)^\bullet$ resolves the structure sheaf $\mathcal{O}_S$ of the intersection $S$ (here we view $\mathcal{O}_S$ as a sheaf on the Grassmannian $G(6,V)$). The $-i^\text{th}$ term of this complex is given below:
\begin{equation*}
\begin{tabular}{*{4}{>{$}c<{$}@{\hskip 0.08in}}}
\cdots \rightarrow & \left( \bigwedge^i \bigwedge^3 (\mathscr{I}_6) \right) \oplus \left( \bigwedge^{i - 1} \bigwedge^3 (\mathscr{I}_6) \otimes \bigwedge^3 (Q^*) \right) \oplus \left( \bigwedge^{i - 2} \bigwedge^3 (\mathscr{I}_6) \otimes \left( \bigwedge^4 (Q^*) \right) ^{\oplus 3} \right) & \rightarrow \cdots. \\
& \scriptstyle{-i} &
\end{tabular}
\end{equation*}
Using now that $H^k(S, \mathcal{O}_S) = H^k(G(6,V),\mathcal{O}_S)$ (see \cite[III, Lem. 2.10]{Hartshorne:1977aa}), the cohomology $H^k(S, \mathcal{O}_S)$ is thus given as the hypercohomology $\mathbb{H}^k \left( G(6,V), \left( \mathbf{F} \otimes \mathbf{HB} \right)^\bullet \right)$ of the tensored complex above. Extending each term $\left( \mathbf{F} \otimes \mathbf{HB} \right)^{-i}$ of this complex to a resolution $\left( \mathbf{F} \otimes \mathbf{HB} \right)^{-i} \rightarrow I^{-i,\bullet}$ of injective sheaves, the vector spaces of global sections of the sheaves $I^{\bullet, \bullet}$ form a double complex whose total cohomology computes the hypercohomology $\mathbb{H}^k \left( G(6,V), \left( \mathbf{F} \otimes \mathbf{HB} \right)^\bullet \right)$ and thus the cohomology $H^k(S, \mathcal{O}_S)$ (see \cite[Lem. 8.5]{Voisin:2002aa}).

This total cohomology is itself computed by a spectral sequence whose $E_0$ page is given by $E_0^{p,q} = I^{p,q}$. To compute the $E_1$ page, we must take cohomology vertically. This amounts to computing the sheaf cohomology of each of the entries $\left( \mathbf{F} \otimes \mathbf{HB} \right)^{-i}$ of the complex above.

Each term of the resolution $\left( \mathbf{F} \otimes \mathbf{HB} \right)^{-i}$, meanwhile, is expressible, using an ``outer plethysm'' decomposition, as a combination of Schur functors of the bundles $\mathscr{E}_6$ and $Q^*$ on $G(6,V)$. For example, the expression
\begin{equation*}\bigwedge^{i} \bigwedge^3 (\mathscr{I}_6) \cong \bigoplus_j \left( \mathbb{S}_{\lambda_j} (\mathscr{E}_6) \right)^{\oplus e_j}\end{equation*}
of $\bigwedge^i \bigwedge^3 (\mathscr{I}_6)$ as a combination of Schur functors of $\mathscr{E}_6$ is obtained through the LiE routine \textit{alt\_tensor} \cite[\S 4.5]{Leeuwen:aa}; more precisely, defining the converter into partition coordinates \textit{to\_eps} as in \cite[\S 5.8.1]{Leeuwen:aa}, the decomposition polynomial $\sum_j e_j X^{\lambda_j}$ of $\bigwedge^{i} \bigwedge^3 (\mathscr{I}_6)$ (see \cite[\S 3.5]{Leeuwen:aa}) is generated by the LiE command \\
\\
\begin{verb}$reverse(to_eps(alt_tensor(i, [0,0,1,0,0,3], A5T1)))$\end{verb},\\
\\
where \textit{reverse} maps $\left( \lambda_{j_1}, \ldots , \lambda_{j_6} \right) \mapsto \left( -\lambda_{j_6}, \ldots , -\lambda_{j_1} \right)$ and extends to polynomials by $\mathbb{Z}$-linearity.

More generally, each term $\left( \mathbf{F} \otimes \mathbf{HB} \right)^{-i}$ of the above complex admits an expression
\begin{equation*}\left( \mathbf{F} \otimes \mathbf{HB} \right)^{-i} \cong \bigoplus_j \left( \mathbb{S}_{\lambda^{(1)}_j} (\mathscr{E}_6) \otimes \mathbb{S}_{\lambda^{(2)}_j} (Q^*) \right)\end{equation*}
for (possibly repeated) weakly decreasing integer sequences $\lambda^{(1)}_j = \left( \lambda^{(1)}_{j_1}, \ldots , \lambda^{(1)}_{j_6} \right)$ and $\lambda^{(2)}_j = \left( \lambda^{(2)}_{j_1}, \ldots , \lambda^{(2)}_{j_4} \right)$. Yet the cohomology of any such summand is determined exhaustively by Bott's theorem (see \cite[\S 2.3, Thm. 8]{Gruson:2013aa}).

This supplies the values of the $E_1$ page of the spectral sequence. Employing the fact that we must have $E_\infty^{p,q} = 0$ whenever $p + q < 0$ or $p + q > 2$, we see that the spectral sequence collapses at the $E_2$ page, revealing the Hodge numbers:
\begin{equation*}h^0(S, \mathcal{O}_S) = 1, \quad h^1(S, \mathcal{O}_S) = 0, \quad h^2(S, \mathcal{O}_S) = 155.\end{equation*}
These agree with the holomorphic Euler characteristic determined by the Chern numbers computed using \cite{Harris:1984aa} above.
\end{example}

\subsection{Products of projective spaces}

We now study pairs $X \subset V$ in which $V$ is a product of projective spaces. Recall that $CH^*(\mathbb{P}^a \times \mathbb{P}^b) \cong \mathbb{Z}[\alpha, \beta] / (\alpha^{a + 1}, \beta^{b + 1})$, where $\alpha$ and $\beta$ are the pullbacks to $\mathbb{P}^a \times \mathbb{P}^b$ of the hyperplane classes on $\mathbb{P}^a$ and $\mathbb{P}^b$, respectively (see \cite[Thm. 2.10]{Eisenbud:2016aa}), and in particular when $a, b \geq 2$, $CH^2(\mathbb{P}^a \times \mathbb{P}^b) \cong \mathbb{Z}^3$ via $\alpha^2$, $\alpha\beta$, and $\beta^2$. Finally, once again $CH^*(\mathbb{P}^a \times \mathbb{P}^b) \rightarrow N^*(\mathbb{P}^a \times \mathbb{P}^b)$ is an isomorphism (see Fulton, MacPherson, Sottile, and Sturmfels \cite[Thm. 2 Cor.]{Fulton:1994aa}). 

Many pairs $X \subset \mathbb{P}^a \times \mathbb{P}^b$ of this kind in fact fail to be decent, in that the quadratic form of Definition \ref{pair} fails to be positive definite (see Remarks \ref{picdecentgen} and \ref{picdecentopp}). This prevents the normal application of Theorem \ref{bound}. On the other hand, various factors permit a partial recovery of the theory. We observe first that the natural identification $N^2(\mathbb{P}^a \times \mathbb{P}^b) \cong \mathbb{Z}^3$ features the additional property that all codimension-2 subvarieties $S' \subset \mathbb{P}^a \times \mathbb{P}^b$ satisfy $[S'] \in \mathbb{N}^3$ in $N^2(\mathbb{P}^a \times \mathbb{P}^b)$. Indeed, this follows from the general theory of toric varieties \cite[Thm. 1 Cor. (i)]{Fulton:1994aa}. (I would like to thank John Ottem for explaining this.) If we consider, instead of all ambient surfaces $S$ in $X$, only those which come directly from a subvariety $S' \subset V$ in the sense of condition \ref{ambient1}. of Definition \ref{ambient}, then we may restrict the function $Q_{X \subset \mathbb{P}^a \times \mathbb{P}^b}$ to the subset $\mathbb{N}^3$ of $\mathbb{Z}^3$. Finally, even when this restriction of $Q_{X \subset \mathbb{P}^a \times \mathbb{P}^b}$ attains arbitrarily negative values, we may often supplement Proposition \ref{euler} with Proposition \ref{alternate}. We illustrate this technique in what follows.

\begin{example}[The product of the quintic threefold and an elliptic curve. Number 41 in \cite{Gray:2013aa}]
The complete intersection in $\mathbb{P}^2 \times \mathbb{P}^4$ of general hypersurfaces of bidegrees $(0,5)$ and $(3,0)$ gives a smooth Calabi--Yau fourfold $X$ which is clearly the product of an elliptic curve and a smooth quintic threefold.

Because $\mathcal{T}_{\mathbb{P}^2 \times \mathbb{P}^4} \cong a^*(\mathcal{T}_{\mathbb{P}^2}) \oplus b^*(\mathcal{T}_{\mathbb{P}^4})$, where $a$ and $b$ say are the projections, pulling back Euler sequences from the factors gives:
\begin{equation*}0 \rightarrow \mathcal{O}_{\mathbb{P}^2 \times \mathbb{P}^4} \oplus \mathcal{O}_{\mathbb{P}^2 \times \mathbb{P}^4} \rightarrow \left( \mathcal{O}_{\mathbb{P}^2 \times \mathbb{P}^4}(1,0) \right)^{\oplus 3} \oplus \left( \mathcal{O}_{\mathbb{P}^2 \times \mathbb{P}^4}(0,1) \right)^{\oplus 5} \rightarrow \mathcal{T}_{\mathbb{P}^2 \times \mathbb{P}^4} \rightarrow 0.\end{equation*}
We also have the following normal bundle sequence:
\begin{equation*}0 \rightarrow \mathcal{T}_X \rightarrow \left. \mathcal{T}_{\mathbb{P}^2 \times \mathbb{P}^4} \right|_X \rightarrow \left. \left( \mathcal{O}_{\mathbb{P}^2 \times \mathbb{P}^4}(0,5) \oplus \mathcal{O}_{\mathbb{P}^2 \times \mathbb{P}^4}(3,0) \right) \right|_X \rightarrow 0.\end{equation*}
Putting these together, we compute the associated function $Q_{X \subset \mathbb{P}^2 \times \mathbb{P}^4}$ on $N^2(\mathbb{P}^2 \times \mathbb{P}^4) \cong \mathbb{Z}^3$:
\begin{equation*}Q_{X \subset \mathbb{P}^2 \times \mathbb{P}^4}(x_1, x_2, x_3) = 30x_2x_3 - 150x_2.\end{equation*}
The second-order part of this quadratic function is not positive definite, and the pair $X \subset \mathbb{P}^2 \times \mathbb{P}^4$ is not decent. We record an example of an application of the available theory:

\begin{theorem}
The product of an elliptic curve and a smooth quintic threefold $X \subset \mathbb{P}^2 \times \mathbb{P}^4$ does not admit a smooth rational ambient surface $S$ in $X$ satisfying $K_S^2 \geq 7$.
\end{theorem}
\begin{proof}
Given any such surface $S$, with $[S] = i^*(\alpha)$ say, Propositions \ref{euler} and \ref{alternate} here give that
\begin{equation*}2 \leq Q_{X \subset \mathbb{P}^2 \times \mathbb{P}^4}(\alpha) \leq 6.\end{equation*}
The divisibility of $Q_{X \subset \mathbb{P}^2 \times \mathbb{P}^4}$ by 30 shows that no such $\alpha$ can exist.
\end{proof}

Thus we recover, in particular, that $X \subset \mathbb{P}^2 \times \mathbb{P}^4$ does not admit an ambient plane.
\end{example}

\begin{example}[Number 130 in \cite{Gray:2013aa}]
\cite[\#130]{Gray:2013aa} gives a smooth Calabi--Yau fourfold $X$ defined as the intersection in $\mathbb{P}^4 \times \mathbb{P}^6$ of hypersurfaces of bidegrees $(0,2), (0,2), (1,1), (1,1), (1,1), (2,0)$.

Proceeding again as above, we compute the associated function $Q_{X \subset \mathbb{P}^4 \times \mathbb{P}^6}$ on $N^2(\mathbb{P}^4 \times \mathbb{P}^6) \cong \mathbb{Z}^3$:
\begin{equation*}Q_{X \subset \mathbb{P}^4 \times \mathbb{P}^6}(x_1, x_2, x_3) = 16x_1x_2 + 48x_1x_3 + 24x_2^2 + 48x_2x_3 + 8x_3^2 - 72x_1 - 128x_2 - 112x_3.\end{equation*}
The second-order part of this quadratic function is not a positive definite form. We again have certain weaker results:

\begin{theorem} \label{130}
Consider the smooth Calabi--Yau fourfold $X \subset \mathbb{P}^4 \times \mathbb{P}^6$ of \cite[\#130]{Gray:2013aa}. Let $r, q \in \mathbb{Z}$. Then at most finitely many elements of $N^2(X)$ are representable by a smooth surface $S$ in $X$ satisfying $\chi(S, \mathcal{O}_S) \leq r$ and $K_S^2 \geq q$ which arises in $X$ as a generically transverse intersection in $\mathbb{P}^4 \times \mathbb{P}^6$.
\end{theorem}
\begin{proof}
By Propositions \ref{euler} and \ref{alternate} and the remarks beginning this section, it suffices to show that the inequality
\begin{equation*}-12r + 2q \leq Q_{X \subset \mathbb{P}^4 \times \mathbb{P}^6}(x_1, x_2, x_3) \leq 6r\end{equation*}
has at most finitely many solutions in $\mathbb{N}^3$. For this it is enough to show that $Q_{X \subset \mathbb{P}^4 \times \mathbb{P}^6}(x_1, x_2, x_3) = p$ has at most finitely many solutions in $\mathbb{N}^3$ for each $p \in \mathbb{Z}$.

Let $p \in \mathbb{Z}$ be arbitrary. Choose $y_2 \geq 5$ so large that $x_2 > y_2$ implies that $24x_2^2 - 128x_2 > p$. Choose $y_3 \geq 3$ so large that $x_3 > y_3$ implies that $8x_3^2 - 112x_3 > p$. Then $Q_{X \subset \mathbb{P}^4 \times \mathbb{P}^6}(x_1, x_2, x_3) > p$ whenever either $x_2 > y_2$ or $x_3 > y_3$. Thus it suffices to count solutions $(x_1, x_2, x_3)$ to $Q_{X \subset \mathbb{P}^4 \times \mathbb{P}^6}(x_1, x_2, x_3) = p$ with $(x_2, x_3) \in \{0, \ldots , y_2\} \times \{0, \ldots , y_3\}$. For each of the finitely many such choices of $(x_2, x_3)$, at most one solution $(x_1, x_2, x_3)$ is accrued. Indeed, a fixed choice of $(x_2, x_3)$ yields a linear equation in $x_1$, which can have infinitely many solutions only perhaps if its first-order coefficient $16x_2 + 48x_3 - 72$ is zero. Reducing this expression modulo 16 shows that this cannot occur.
\end{proof}
\end{example}

\begin{example}[Number 133 in \cite{Gray:2013aa}]
\cite[\#133]{Gray:2013aa} gives a smooth Calabi--Yau fourfold $X$ defined as the smooth complete intersection in $\mathbb{P}^5 \times \mathbb{P}^5$ of hypersurfaces of bidegrees $(0,2), (0,2), (1,1), (1,1), (2,0), (2,0)$.

As above, we have the associated function $Q_{X \subset \mathbb{P}^5 \times \mathbb{P}^5}$ on $N^2(\mathbb{P}^5 \times \mathbb{P}^5) \cong \mathbb{Z}^3$ as follows:
\begin{equation*}Q_{X \subset \mathbb{P}^5 \times \mathbb{P}^5}(x_1, x_2, x_3) = 32x_1x_2 + 64x_1x_3 + 32x_2^2 + 32x_2x_3 - 96x_1 - 128x_2 - 96x_3.\end{equation*}
We record a result even weaker than that given above:
\begin{theorem}
Consider the smooth Calabi--Yau fourfold $X \subset \mathbb{P}^5 \times \mathbb{P}^5$ of \cite[\#133]{Gray:2013aa}. Let $r, q \in \mathbb{Z}$ be such that $-12r + 2q > -96$. Then at most finitely many elements of $N^2(X)$ are representable by a smooth surface $S$ in $X$ satisfying $\chi(S, \mathcal{O}_S) \leq r$ and $K_S^2 \geq q$ which arises in $X$ as a generically transverse intersection in $\mathbb{P}^5 \times \mathbb{P}^5$.
\end{theorem}
\begin{proof}
We show as before that, for $r, q$ chosen as above,
\begin{equation*}-12r + 2q \leq Q_{X \subset \mathbb{P}^5 \times \mathbb{P}^5}(x_1, x_2, x_3) \leq 6r\end{equation*}
has at most finitely many solutions in $\mathbb{N}^3$. For this it is enough to show that $Q_{X \subset \mathbb{P}^5 \times \mathbb{P}^5}(x_1, x_2, x_3) = p$ has at most finitely many solutions in $\mathbb{N}^3$ for each $p > -96$.

The divisibility of $Q_{X \subset \mathbb{P}^5 \times \mathbb{P}^5}$ by $32$ shows that it suffices to consider $Q_{X \subset \mathbb{P}^5 \times \mathbb{P}^5}(x_1, x_2, x_3) = p$ for the values $p = -64, -32, 0, \ldots$. We claim that each such choice of $p$ yields an equation with only finitely many solutions.

We factor the equation $Q_{X \subset \mathbb{P}^5 \times \mathbb{P}^5}(x_1, x_2, x_3) = -64$, writing instead
\begin{equation*}(2x_1 + x_2 - 3)(2x_3 + x_2 - 3) = -x_2^2 + 2x_2 + 5.\end{equation*}
If $x_2 > 3$, then the left-hand side of the above equation is nonnegative while the right-hand side is negative, so we consider only solutions $(x_1, x_2, x_3)$ with $x_2 \in \{0, 1, 2, 3\}$. Each such choice of $x_2$ yields a hyperbolic equation in $x_1$ and $x_3$:
\begin{equation*}
\begin{array}{cc} 
\begin{aligned}
x_2 = 0: \quad \\
x_2 = 1: \quad \\
x_2 = 2: \quad \\
x_2 = 3: \quad \\
\end{aligned} &
\begin{aligned}
(2x_1 - 3)(2x_3 - 3) &= 5, \\
(2x_1 - 2)(2x_3 - 2) &= 6, \\
(2x_1 - 1)(2x_3 - 1) &= 5, \\
(2x_1)(2x_3) &= 2. \\
\end{aligned}
\end{array}
\end{equation*}

The solutions of each such equation are enumerated by counting the factorizations of the right-hand constant term into two factors of the correct parity. Each such constant term admits at most finitely many such factorizations, unless, perhaps, that term is zero and the hyperbola degenerates to a product of two lines. This does not occur here, and we get 4 solutions to $Q_{X \subset \mathbb{P}^5 \times \mathbb{P}^5}(x_1, x_2, x_3) = -64$ in $\mathbb{N}^3$. (I would like to thank David Savitt for this argument.)

Incrementing $p$ by 32 increases each of the constant terms as above by 2, with the caveat that the bound $x_2 \in \{0, 1, 2, 3\}$ may cease to hold and we may be forced to consider higher values of $x_2$. On the other hand, when $x_2 > 3$, the left-hand side of the factored equation above is not just nonnegative but positive, and we again need not fear the degeneration of the hyperbola into lines.
\end{proof}

Decrementing $p$ by 32 decreases each of the constant terms as above by 2, yielding for the lowermost equation the infinite families of solutions $(x_1, 3, 0), x_1 \geq 0$ and $(0, 3, x_3), x_3 \geq 0$. Thus we see that the above result is optimal.

\begin{remark}
Number-theoretic techniques which much more sophisticatedly determine or estimate the number of lattice points in situations such as those above exist. See, for example, Kr\"{a}tzel \cite{Kratzel:1988aa}.
\end{remark}
\end{example}

\bibliography{bibliography} {}

\begin{thebibliography}{10}

\bibitem{Barth:2004aa}
W.~P. Barth, K.~Hulek, C.~A.~M. Peters, and A.~Van~de Ven.
\newblock {\em Compact Complex Surfaces}, volume~4 of {\em A Series of Modern
  Surveys in Mathematics}.
\newblock Springer, second enlarged edition, 2004.

\bibitem{Beauville:1983aa}
A.~Beauville.
\newblock Vari{\'e}t{\'e}s {K}{\"a}hleriennes dont la premi{\`e}re classe de
  {C}hern est nulle.
\newblock {\em J. Differential Geometry}, 18(4):755--782, 1983.

\bibitem{Beauville:2007aa}
A.~Beauville.
\newblock On the splitting of the {B}loch--{B}eilinson filtration.
\newblock In J.~Nagel and C.~Peters, editors, {\em Algebraic Cycles and
  Motives}, volume~2 of {\em London Mathematical Society Lecture Note Series}.
  Cambridge University Press, 2007.

\bibitem{Beauville:1985aa}
A.~Beauville and R.~Donagi.
\newblock La vari{\'e}t{\'e} des droites d'une hypersurface cubique de
  dimension 4.
\newblock {\em C. R. Acad. Sc. Paris}, 301(14):703--706, 1985.

\bibitem{Bloch:1976aa}
S.~Bloch.
\newblock An example in the theory of algebraic cycles.
\newblock In M.~R. Stein, editor, {\em Algebraic K-Theory}. Springer, 1976.

\bibitem{Bloch:1974aa}
S.~Bloch and A.~Ogus.
\newblock Gersten's conjecture and the homology of schemes.
\newblock {\em Ann. Scient. {\'E}c. Norm. Sup.}, 7(2):181--202, 1974.

\bibitem{Bloch:1983aa}
S.~Bloch and V.~Srinivas.
\newblock Remarks on correspondences and algebraic cycles.
\newblock {\em Am. J. Math.}, 105(5):1235--1253, 1983.

\bibitem{Brunner:1997aa}
I.~Brunner, M.~Lynker, and R.~Schimmrigk.
\newblock Unification of {M}- and {F}-theory {C}alabi--{Y}au fourfold vacua.
\newblock {\em Nucl. Phys. B}, 498:156--174, 1997.

\bibitem{Candelas:1991aa}
P.~Candelas, X.~C.~De La~Ossa, P.~S. Green, and L.~Parkes.
\newblock A pair of {C}alabi--{Y}au manifolds as an exactly soluble
  superconformal theory.
\newblock {\em Nucl. Phys. B}, 359:21--74, 1991.

\bibitem{Charles:2013aa}
F.~Charles and E.~Markman.
\newblock The standard conjectures for holomorphic symplectic varieties
  deformation equivalent to {H}ilbert schemes of {K}3 surfaces.
\newblock {\em Compositio Math.}, 149(481-494), 2013.

\bibitem{Ciliberto:2002aa}
C.~Ciliberto and V.~Di~Gennaro.
\newblock Boundedness for surfaces on smooth fourfolds.
\newblock {\em J. Pure Appl. Algebr.}, 173:273--279, 2002.

\bibitem{Cohn:1980aa}
H.~Cohn.
\newblock {\em Advanced Number Theory}.
\newblock Dover Publications, 1980.

\bibitem{Debarre:2010aa}
O.~Debarre and C.~Voisin.
\newblock Hyper-{K}{\"a}hler fourfolds and {G}rassmann geometry.
\newblock {\em J. Reine Angew. Math.}, 649:63--87, October 2010.

\bibitem{Dinh:2006aa}
T.-C. Dinh and V.-A. Nguy{\^e}n.
\newblock The mixed {H}odge--{R}iemann bilinear relations for compact
  {K}{\"a}hler manifolds.
\newblock {\em Geom. Funct. Anal.}, 16:838--849, 2006.

\bibitem{Eisenbud:1995aa}
D.~Eisenbud.
\newblock {\em Commutative Algebra with a View Toward Algebraic Geometry},
  volume 150 of {\em Graduate Texts in Mathematics}.
\newblock Springer-Verlag, 1995.

\bibitem{Eisenbud:2005aa}
D.~Eisenbud.
\newblock {\em The Geometry of Syzygies}.
\newblock Springer, 2005.

\bibitem{Eisenbud:2016aa}
D.~Eisenbud and J.~Harris.
\newblock {\em 3264 \& All That}.
\newblock Cambridge University Press, 2016.

\bibitem{Ellingsrud:1989aa}
G.~Ellingsrud and C.~Peskine.
\newblock Sur les surfaces lisses de $\mathbb{P}^4$.
\newblock {\em Invent. Math.}, 95:1--11, 1989.

\bibitem{Ferretti:2012aa}
A.~Ferretti.
\newblock The {C}how ring of double {EPW} sextics.
\newblock {\em Algebra Number Theory}, 6:539--560, 2012.

\bibitem{Fulton:1984aa}
W.~Fulton.
\newblock {\em Intersection Theory}.
\newblock Springer, 1984.

\bibitem{Fulton:1994aa}
W.~Fulton, R.~MacPherson, F.~Sottile, and B.~Sturmfels.
\newblock Intersection theory on spherical varieties.
\newblock {\em J. Algebraic Geometry}, 4:181--193, 1994.

\bibitem{Gray:2013aa}
J.~Gray, A.~S. Haupt, and A.~Lukas.
\newblock All complete intersection {C}alabi--{Y}au four-folds.
\newblock {\em J. High Energy Phys.}, 1307(70), 2013.

\bibitem{Greene:1995aa}
B.~R. Greene, D.~R. Morrison, and M.~R. Plesser.
\newblock Mirror manifolds in higher dimension.
\newblock {\em Commun. Math. Phys.}, 173(3):559--597, 1995.

\bibitem{Gruson:2013aa}
L.~Gruson, S.~V. Sam, and J.~Weyman.
\newblock Moduli of abelian varieties, {V}inberg theta-groups, and free
  resolutions.
\newblock In I.~Peeva, editor, {\em Commutative Algebra: Expository Papers
  Dedicated to David Eisenbud on the Occasion of His 65th Birthday}. Springer,
  2013.

\bibitem{Harris:1984aa}
J.~Harris and L.~Tu.
\newblock Chern numbers of kernel and cokernel bundles.
\newblock {\em Invent. Math.}, 75:467--476, 1984.

\bibitem{Hartshorne:1977aa}
R.~Hartshorne.
\newblock {\em Algebraic Geometry}, volume~52 of {\em Graduate Texts in
  Mathematics}.
\newblock Springer-Verlag, 1977.

\bibitem{Huybrechts:2016aa}
D.~Huybrechts.
\newblock {\em Lectures on {K}3 Surfaces}.
\newblock Number 158 in Cambridge Studies in Advanced Mathematics. Cambridge
  University Press, 2016.

\bibitem{Jannsen:2000aa}
U.~Jannsen.
\newblock Equivalence relations on algebraic cycles.
\newblock In B.~B. Gordon, J.~D. Lewis, S.~M{\"u}ller-Stach, S.~Shuji, and
  N.~Yui, editors, {\em The Arithmetic and Geometry of Algebraic Cycles}.
  Springer Science and Business Media, 2000.

\bibitem{Kleiman:1971aa}
S.~Kleiman.
\newblock Les theoremes de finitude pour le foncteur de {P}icard.
\newblock In A.~Dold and B.~Eckmann, editors, {\em Th{\'e}orie des
  Intersections et Th{\'e}or{\`e}me de Riemann-Roch}, volume 225 of {\em
  Lecture Notes in Mathematics}. Springer-Verlag, 1971.

\bibitem{Kollar:1999aa}
J.~Koll{\'a}r.
\newblock {\em Rational Curves on Algebraic Varieties}.
\newblock Springer, 1999.

\bibitem{Kratzel:1988aa}
E.~Kr{\"a}tzel.
\newblock {\em Lattice Points}.
\newblock Kluwer Academic Publishers, 1988.

\bibitem{Laterveer:2015aa}
R.~Laterveer.
\newblock Yet another version of {M}umford's theorem.
\newblock {\em Arch. Math.}, 104(2):125--131, 2015.

\bibitem{Lawrence:2014aa}
J.~D. Lawrence.
\newblock {\em A Catalogue of Special Plane Curves}.
\newblock Dover Publications, 2014.

\bibitem{Lewis:1989aa}
J.~D. Lewis.
\newblock Towards a generalization of {M}umford's theorem.
\newblock {\em J. Math. Kyoto Univ.}, 29(2):267--272, 1989.

\bibitem{Mumford:1985aa}
D.~Mumford.
\newblock {\em Abelian Varieties}.
\newblock Tata Institute of Fundamental Research Studies in Mathematics. Oxford
  University Press, second edition, 1985.

\bibitem{Murre:1994aa}
J.~P. Murre.
\newblock Algebraic cycles and algebraic aspects of cohomology and {K}-theory.
\newblock In A.~Albano and F.~Bardelli, editors, {\em Algebraic Cycles and
  Hodge Theory}, Lecture Notes in Mathematics, pages 93--152. Springer, 1994.

\bibitem{Timorin:1998aa}
V.~A. Timorin.
\newblock Mixed {H}odge--{R}iemann bilinear relations in a linear context.
\newblock {\em Funct. Anal. Appl.}, 32(4):268--272, 1998.

\bibitem{Leeuwen:aa}
M.~A.~A. van Leeuwen, A.~M. Cohen, and B.~Lisser.
\newblock {\em {LiE} Manual}.
\newblock Computer Algebra Group of CWI, Kruislaan 413, 1098 SJ Amsterdam, The
  Netherlands, 2.2.2 edition.

\bibitem{Voisin:2002aa}
C.~Voisin.
\newblock {\em Hodge Theory and Complex Algebraic Geometry {I}}, volume~76 of
  {\em Cambridge Studies in Advanced Mathematics}.
\newblock Cambridge University Press, 2002.

\bibitem{Voisin:2003aa}
C.~Voisin.
\newblock {\em Hodge Theory and Complex Algebraic Geometry {II}}, volume~77 of
  {\em Cambridge Studies in Advanced Mathematics}.
\newblock Cambridge University Press, 2003.

\bibitem{Voisin:2004aa}
C.~Voisin.
\newblock Intrinsic pseudo-volume forms and {K}-correspondences.
\newblock In A.~Collino, A.~Conte, and M.~Marchisio, editors, {\em The Fano
  Conference}, pages 761--792. Universit{\`a} di Torino, Dipartimento di
  Matematica, 2004.

\bibitem{Voisin:2008aa}
C.~Voisin.
\newblock On the chow ring of certain algebraic hyper-{K}{\"a}hler manifolds.
\newblock {\em Pure. Appl. Math. Q.}, 4(3):613--649, 2008.

\bibitem{Voisin:2014aa}
C.~Voisin.
\newblock {\em Chow Rings, Decomposition of the Diagonal, and the Topology of
  Families}, volume 187 of {\em Annals of Mathematics Studies}.
\newblock Princeton University Press, 2014.

\bibitem{Voisin:2016aa}
C.~Voisin.
\newblock Remarks and questions on coisotropic subvarieties and 0-cycles of
  hyper-{K}{\"a}hler varieties.
\newblock In C.~Faber, G.~Farkas, and G.~van~der Geer, editors, {\em K3
  Surfaces and Their Moduli}, volume 315 of {\em Progress in Mathematics}.
  Birkhauser, 2016.

\end{thebibliography}
\bibliographystyle{abbrv}

\textsc{Benjamin E. Diamond}
\newline Johns Hopkins University,
\newline Department of Mathematics,
\newline 3400 N. Charles Street,
\newline Baltimore, MD 21218\par\nopagebreak
email: \texttt{bdiamond@math.jhu.edu}

\end{document}